\documentclass[reqno]{amsart}
\usepackage{newcent}       % selects Times Roman as basic font
\usepackage{helvet}         % selects Helvetica as sans-serif font
\usepackage{courier}        % selects Courier as typewriter font
\usepackage{amsfonts}
\usepackage{amsfonts,amssymb,amsmath}
\usepackage[latin1]{inputenc}
\usepackage{color}
\usepackage{graphicx}
\usepackage{amsmath}
\usepackage{mathrsfs}
\usepackage[mathscr]{euscript}
\usepackage[sans]{dsfont}
\usepackage{tikz}
\usepackage{bbm}
\usepackage[small]{caption}

\makeatletter
\@addtoreset{equation}{section}
\makeatother

\newtheorem{theorem}{Theorem}[section]
\newtheorem{lemma}[theorem]{Lemma}
\newtheorem{proposition}[theorem]{Proposition}
\newtheorem{corollary}[theorem]{Corollary}

\newtheorem{remark}[theorem]{Remark}
\newtheorem{example}[theorem]{Example}

\newcommand{\mc}[1]{{\mathcal #1}}
\newcommand{\mf}[1]{{\mathfrak #1}}
\newcommand{\mb}[1]{{\mathbf #1}}
\newcommand{\bb}[1]{{\mathbb #1}}
\newcommand{\ms}[1]{{\mathscr #1}}
\newcommand{\bs}[1]{{\boldsymbol #1}}

\newcounter{as}[section]

\newcommand{\<}{\langle}
\renewcommand{\>}{\rangle}
\renewcommand{\Cap}{{\rm cap}}

\title[Metastable Markov chains]{Metastable Markov chains: from the
  convergence of the trace to the convergence of the
  finite-dimensional distributions}

\author{C. Landim, M. Loulakis, M. Mourragui}

\address{\noindent IMPA, Estrada Dona Castorina 110, CEP 22460 Rio de
  Janeiro, Brasil and CNRS UMR 6085, Universit\'e de Rouen, Avenue de
  l'Universit\'e, BP.12, Technop\^ole du Madril\-let, F76801
  Saint-\'Etienne-du-Rouvray, France.}
\email{landim@impa.br}

\address{School of Applied Mathematical and Physical
  Sciences, National Technical University of Athens, 15780 Athens,
  Greece and Institute of Applied and Computational Mathematics, Foundation for Research and Technology- Hellas (FORTH), 70013 Heraklion Crete, Greece. }
\email{\tt loulakis@math.ntua.gr}

\address{
  \hfill\break\indent CNRS UMR 6085, Universit\'e de Rouen,
  \hfill\break\indent Avenue de l'Universit\'e, BP.12, Technop\^ole du
  Madril\-let, \hfill\break\indent
F76801 Saint-\'Etienne-du-Rouvray, France.}
\email{Mustapha.Mourragui@univ-rouen.fr}

\begin{document}

\begin{abstract}
  We consider continuous-time Markov chains which display a family of
  wells at the same depth. We provide sufficient conditions which
  entail the convergence of the finite-dimensional distributions of
  the order parameter to the ones of a finite state Markov chain. We
  also show that the state of the process can be represented as a
  time-dependent convex combination of metastable states, each of
  which is supported on one well. 
\end{abstract} 

\maketitle

\section{Introduction}
\label{sec01}

Several different methods to prove the metastable behavior of Markov
chains have been proposed in the last years \cite{ov05, BovHol, cn13,
  CNS2015, fmns15, biagau2016, fmnss16}. 

Inspired by the potential theoretic approach to metastability,
proposed by Bovier, Eckhoff, Gayrard and Klein in \cite{begk01,
  begk02}, Beltr\'an and Landim introduced a general method, known as
the martingale method, to derive the metastable behavior of a Markov
process \cite{bl2, bl4, bl9}. The reader will find in \cite{bl9} a
discussion on the similarities and differences between the martingale
approach, the pathwise approach, put forward in \cite{cgov} and
presented in \cite{ov05}, and the potential theoretic approach,
proposed in \cite{begk01, begk02} and reviewed in \cite{BovHol}.

To insert the main results of the article in their context, we
recall below the martingale method in the context of condensing
zero-range processes \cite{bl2012a, l2014, s2018}.  Denote by $\bb N$
the non-negative integers, $\bb N =\{0,1,2,...\}$, by $\bb T_L$, $L\ge
1$, the discrete, one-dimensional torus with $L$ points, and by $\eta$
the elements of $\bb N^{\bb T_L}$ called configurations. The total
number of particles at $x\in \bb T_L$ for a configuration $\eta \in
\bb N^{\bb T_L}$ is represented by $\eta_x$. Let $E_N$, $N\ge 1$, be
the set of configurations with $N$ particles:
\begin{equation}
\label{f02}
E_{N}\;:=\;\big\{ \eta\in\bb N^{\bb T_L} : \sum_{x\in \bb T_L} \eta_x
= N \big\}\;.
\end{equation}

Fix $\alpha>1$, and define $g:\bb N\to \bb R_+$ as
\begin{equation*}
g(0)=0\; , \quad g(1)=1  \quad
\textrm{and}\quad g(n)=\frac{a(n)}{a(n-1)}\;, \;\; n\ge 2\;,  
\end{equation*}
where $a(0)=1$, $a(n)=n^{\alpha}$, $n\ge 1$.  In this way,
$\prod_{i=1}^{n}g(i)=a(n)$, $n\ge 1$, and $\{ g(n) : n\ge 2\}$ is a
strictly decreasing sequence converging to $1$ as $n\uparrow\infty$.

Fix $1/2\le p\le 1$, and denote by $p(x)$ the transition probability
given by $p(1)=p$, $p(-1)=1-p$, $p(x)=0$, otherwise.  Let
$\sigma^{x,y}\eta$ be the configuration obtained from $\eta$ by moving
a particle from $x$ to $y$:
\begin{equation}
\label{f01}
(\sigma^{x,y}\eta)_z\;=\;\left\{
\begin{array}{ll}
\eta_x-1 & \textrm{for $z=x$} \\
\eta_y+1 & \textrm{for $z=y$} \\
\eta_z & \rm{otherwise}\;. \\
\end{array}
\right.  
\end{equation}

The nearest-neighbor, zero-range process associated to the jump rates
$\{g(k) : k\ge 0\}$ and the transition probability $p(x)$ is the
continuous-time, $E_N$-valued Markov process $\{\eta^N(t) : t\ge 0\}$
whose generator $L_N$ acts on functions $f: E_N\to\bb R$ as
\begin{equation*}
(L_N f) (\eta) \;=\; \sum_{\stackrel{x,y\in \bb T_L}{x\not = y}}
g(\eta_x) \, p(y-x) \, \big\{ f(\sigma^{x,y}\eta) - f(\eta) \big\} \;.
\end{equation*}
Hence, if there are $k$ particles at site $x$, at rate $p g(k)$, resp.
$(1-p) g(k)$, one of them jumps to the right, resp. left. Since $g(k)$
decreases to $1$ as $k\to\infty$, the more particles there are at some
site $x$ the slower they jump, but the rate remains bounded below by
$1$. 

This Markov process is irreducible. The stationary probability
measure, denoted by $\mu_N$, is given by
\begin{equation*}
\mu_N(\eta) \;=\; \frac {N^{\alpha}} {Z_{N}} \, 
\prod_{x\in \bb T_L} \frac{1}{ a(\eta_x)} \;,
\end{equation*}
where $Z_{N}$ is the normalizing constant.

Fix a sequence $\{\ell_N : N\ge 1\}$ such that $\ell_N \to \infty$,
$N/\ell_N \to \infty$, and let $\ms E^x_N$, $x\in \bb T_L$, be the set
of configurations in which all but $\ell_N$ particles sit at $x$:
\begin{equation*}
\ms E^x_N  \;:=\; \Big\{\eta\in E_N : \eta_x \ge N - \ell_N \Big\}\;.
\end{equation*}
According to equation (3.2) in \cite{bl2012a}, for each $x\in \bb
T_L$, $\mu_N(\ms E^x_N) \to 1/L$ as $N\uparrow\infty$.

By the ergodic theorem, the process stays most of the time in the set
$\sqcup_{x\in\bb T_L} \ms E^x_N$. Since these sets are far apart, one
expects the sets $\ms E^x_N$ to behave as wells of the dynamics: the
process remains for a very long time in each of the sets $\ms E^x_N$
at the end of which it performs a quick transition to another set $\ms
E^y_N$.

If the process evolves as described in the previous paragraph, it is
reasonable to call depth of the well $\ms E^x_N$ the average time the
process remains in $\ms E^x_N$ before hitting another well. The
symmetry of the model implies that in the zero-range process
introduced above all wells have the same depth. This is an important
difference between this dynamics and the previous ones in which a
metastable behavior has been observed. In the latter ones,
cf. \cite{ov05, BovHol}, the models feature one shallow and one deep
well and the problem consists in describing the transition from the
shallow well to the deep one, or in estimating the mean value of the
transition time. In contrast, in the zero-range process, the presence
of many wells of the same depth transforms the problem in the
characterization of the evolution of the process among the wells.

Beltr\'an and Landim proposed in \cite{bl2, bl4} a mathematical
formulation of this phenomenon which we present below in the context
of a sequence of Markov chains, each of which takes values in a finite
set.

Consider a sequence of finite sets $(E_N : N\ge 1)$ whose cardinality
tends to infinity with $N$. The elements of $E_N$ are called
configurations and are denoted by the Greek letters $\eta$, $\xi$,
$\zeta$. Let $\{\eta^N(t) : t\ge 0\}$ be a continuous-time,
$E_N$-valued, irreducible Markov chain.

\smallskip\noindent{\bf The wells.}  Consider a partition $\ms E^1_N,
\dots, \ms E^{\mf n}_N$, $\Delta_N$, $\mf n\ge 2$, of the set $E_N$,
and let
\begin{equation}
\label{17}
%E_N \;=\; \ms E_N \sqcup \Delta_N \;, \quad
\ms E_N \;=\; \ms E^1_N \sqcup  \cdots \sqcup \ms E^{\mf n}_N
\;, \quad  \breve{\ms E}^x_N \;=\; \bigsqcup_{y\not = x} \ms E^y_N\;.
\end{equation}
Here and below we use the notation $\ms A\sqcup \ms B$ to represent
the union of two disjoint sets $\ms A$, $\ms B$: $\ms A\sqcup \ms B =
\ms A\cup \ms B$, and $\ms A\cap \ms B=\varnothing$. As in the example
above, the sets $\ms E^x_N$ have to be understood as the wells of the
dynamics, the sets where the process remains most of the time, and
$\Delta_N$ as the set which separates the wells.

\smallskip\noindent{\bf The time scale.} Let $(\theta_N: N\ge 1)$ be
the time-scale at which one observes a transition from a well $\ms
E^x_N$ to the set $\breve{\ms E}^x_N$ which consists of the union of
all the other wells. This time-scale has to be determined in each
model. As it can be expressed in terms of capacities (cf. Lemma 6.8 in
\cite{bl2}), its derivation corresponds to the calculation of the
capacity between $\ms E^x_N$ and $\breve{\ms E}^x_N$.

Denote by $\xi^N(t)$ the process $\eta^N(t)$ speeded-up by $\theta_N$:
$\xi^N(t) = \eta^N(t\theta_N)$. Note that the transitions between
wells occur in time-intervals of order $1$ for the process $\xi^N(t)$.
This is the reason for changing the time scale and introducing
$\xi^N(t)$.

\smallskip\noindent{\bf Model reduction.} We expect the process to
remain for a very long time in each well, a time much longer than the time
it needs to equilibrate inside the well. If this description is
correct, the hitting time of a new well should be asymptotically
Markovian due to the loss of memory entailed by the equilibration.

Let $\Phi_N: E_N \to \{0,1, \dots, \mf n\}$, $\Psi_N: \ms E_N \to \{1,
\dots, \mf n\}$ be the projections defined by
\begin{equation*}
\Phi_N (\eta) \;=\; \sum_{x=1}^{\mf n} x \, \mb 1\{\eta
\in \ms E^x_N\} \;, \qquad
\Psi_N (\eta) \;=\; \sum_{x=1}^{\mf n} x \, \mb 1\{\eta
\in \ms E^x_N\} \;.
\end{equation*}
Note that $\Phi_N (\eta) =0$ for $\eta\in \Delta_N$, while $\Psi_N$ is
not defined on the set $\Delta_N$. In general, $\Phi_N(\xi^N(t))$ is
not a Markov chain, but only a hidden Markov chain.  As the
cardinality of $E_N$ increases to $\infty$ with $N$,
$\Phi_N(\xi^N(t))$ takes values in a much smaller state space than
$\xi^N(t)$. For this reason it is called the reduced chain.

The argument laid down above on equilibration and loss of memory
suggests that $\Phi_N(\xi^N(t))$ converges to a Markov chain taking
values in $\{0,1, \dots, \mf n\}$. However, the brief sojourns at
$\Delta_N$ create an obstacle to the convergence. Starting from the
well $\ms E^x_N$, the process $\xi^N(t)$ makes many unsuccessful
attempts before hitting a new well $\ms E^y_N$. These attempts
correspond to brief visits to $\Delta_N$. A typical path of
$\Phi_N(\xi^N(t))$ is illustrated in Figure \ref{fig2}. These short
sojourns at $\Delta_N$, which disappear in the limit, prevent the
convergence (in the usual Skorohod topology) of the process
$\Phi_N(\xi^N(t))$ to a $\{1, \dots, \mf n\}$-valued Markov chain.

To overcome this difficulty, we perform a small surgery in the
trajectories by removing from them the pieces of the paths in
$\Delta_N$. This is done by considering the trace of the process
$\xi^N(t)$ on $\ms E_N$.

\begin{figure}
  \centering
\begin{tikzpicture}[scale = .5]
\draw[->] (-.2, 0) -- (20,0) node[right] {$t$};
\draw[->] (0,-.2) -- (0,5);
\foreach \y in {0, ..., 3}
\draw (-.1,1+\y) -- (0,1+\y) node[left] {$\y$};
\draw (0, 2) -- (8,2);
\fill[black!100] (8,1) -- (8.2,1) -- (8.2,2) -- (8,2) -- (8,1);
\draw (8.2, 2) -- (14,2);
\fill[black!100] (14,1) -- (14.2,1) -- (14.2,2) -- (14,2) -- (14,1);
\draw (14.2,2) --(14.2,4);
\draw (14.2, 4) -- (20,4);
\end{tikzpicture}
\caption{A typical trajectory of $\Phi_N(\xi^N(t))$. The black
  rectangles represent jumps from $1$ to $0$ and from $0$ to $1$ in
  short time intervals when the process reaches the boundary of the
  well $\ms E^1_N$.}
\label{fig2}
\end{figure}
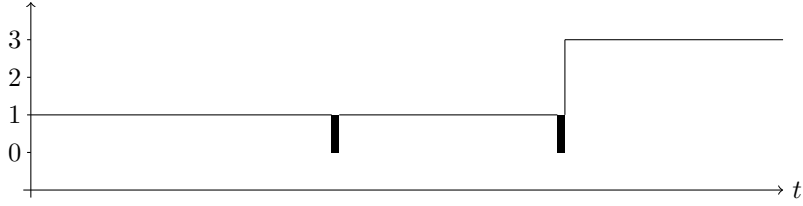

\smallskip\noindent{\bf Trace process.}  Fix a proper subset $\ms A$
of $E_N$. The trace of the process $\xi^N(t)$ on the set $\ms A$,
denoted by $\xi^{\ms A}(t)$, is the process obtained from $\xi^N(t)$
by stopping its evolution when it leaves the set $\ms A$ and by
restarting it when it returns to the set $\ms A$. More precisely,
denote by $T_{\ms A}(t)$ the total time spent at $\ms A$ before time
$t$:
\begin{equation*}
T_{\ms A}(t) \; =\; \int_0^t \mb 1\{\xi^N(s) \in \ms A\}\, ds\;,
\end{equation*}
where $\mb 1\{B\}$ represents the indicator of the set $B$. Note that
the function $T_{\ms A}$ is piecewise differentiable and that its
derivative takes only the values $1$ and $0$. It is equal to $1$ when
the process is in $\ms A$ and it is equal to $0$ when it is not. Let
$S_{\ms A}(t)$ be the generalized inverse of $T_{\ms A}(t)$:
\begin{equation*}
S_{\ms A}(t) = \sup \{s\ge 0: T_{\ms A}(s) \le t\}\;.
\end{equation*}
The trace process is defined as $\xi^{\ms A}(t) = \xi^N(S_{\ms
  A}(t))$.  It is shown in \cite[Proposition 6.1]{bl2} that if
$\xi^N(t)$ is a continuous-time, irreducible Markov chain, then
$\xi^{\ms A}(t)$ is a continuous-time, $\ms A$-valued, irreducible
Markov chain whose jump rates can be expressed in terms of the
probabilities of hitting times of the original chain.

Denote by $\xi^{\ms E_N}(t)$ the trace of the process $\xi^N(t)$ on
$\ms E_N$. By the previous paragraph, $\xi^{\ms E_N}(t)$ is an $ \ms
E_N$-valued Markov process. If the time spent on $\Delta_N$ is
negligible, we only removed from the original trajectory the short
sojourns in $\Delta_N$.

\smallskip\noindent{\bf Metastability.}  Denote by $X_N(t)$,
$X^T_N(t)$ the hidden Markov chains given by $X_N(t) =
\Phi_N(\xi^N(t))$, $X^T_N(t) = \Psi_N(\xi^{\ms E_N}(t))$,
respectively. Note that $X_N(t)$ takes values in $\{0, 1, \dots, \mf
n\}$, while $X^T_N(t)$ takes values on the set $S:=\{1, \dots, \mf
n\}$.  Moreover, $X^T_N(t)$ is the trace on the set $S$ of the process
$X_N(t)$.

Let $D(\bb R_+, E_N)$ be the space of right-continuous functions
$\omega: \bb R_+ \to E_N$ with left-limits endowed with the Skorohod
topology. Let $\bb P_\eta=\bb P^N_\eta$, $\eta\in E_N$, be the
probability measure on the path space $D(\bb R_+, E_N)$ induced by the
Markov chain $\xi^N(t)$ starting from $\eta$. Expectation with respect
to $\bb P_\eta$ is represented by $\bb E_\eta$.

In \cite{bl2, bl4, bl9}, a set of conditions have been introduced
which yield that

\begin{enumerate}
\item[({\bf H1})] The dynamics $X^T_N(t) = \Psi_N(\xi^{\ms E_N}(t))$
  is asymptotically Markovian: For all $x\in S$, and sequences
  $\eta^N\in\ms E^x_N$, under the measure $\bb P_{\eta^N}$ the
  process $X^T_N(t)$ converges in the Skorohod topology to a Markov
  chain denoted by $\bs X(t)$;
\item[({\bf H2})] The time spent in $\Delta_N$ is negligible: For all $t>0$
\begin{equation*}
\lim_{N\to\infty}
\max_{\eta\in\ms E_N} \bb E_{\eta} \Big[
\int_0^t \mb 1\{X_N(s) =0 \}\, ds\Big] \;=\; 0\;. 
\end{equation*}
\end{enumerate}

The first condition asserts that the trace on $S$ of the process
$X_N(t)$ converges to a Markov chain, while the second one states that
the amount of time the process $X_N(t)$ spends outside $S$ vanishes as
$N\uparrow\infty$, uniformly over initial configurations in $\ms E_N$.

The second condition can be restated as
\begin{equation}
\label{01}
\lim_{N\to\infty}
\max_{\eta\in\ms E_N} \bb E_{\eta} \Big[
\int_0^t \mb 1\{\xi^N(s) \in \Delta_N\}\, ds\Big] \;=\; 0\;. 
\end{equation}

\smallskip\noindent{\bf Soft topology.} It is clear that the
convergence of the process $X_N(t)$ to $\bs X(t)$ in the Skorohod
topology does not follow from conditions (H1) and (H2). Consider, for
example, a continuous-time, $S$-valued Markov chain $Y(t)$, and
a sequence $\delta_N>0$, $\delta_N\downarrow 0$. Fix $t_0>0$, and define
the process $Y_N(t)$ by $Y_N(t)=Y(t) \mb 1\{t\not \in [t_0-\delta_N,
t_0+\delta_N)\}$. The sequence of processes $Y_N(t)$ fulfills
properties (H1) and (H2), but $Y_N(t)$ does not converge to $Y(t)$ in
the Skorohod topology. Actually, not even the $1$-dimensional
distributions converge.

This example is artificial, but in almost all models in which a
metastable behavior has been observed (cf. the examples of Section
\ref{sec02}), as mentioned in the subsection Model Reduction, due to
the many and very short sojourns of $X_N(t)$ in $0$, the process
$X_N(t)$ can not converge in any of the Skorohod topologies to $\bs
X(t)$. To overcome this obstacle a weaker topology has been proposed
in \cite{l1}, called the soft topology, in which the convergence takes
place.

The soft topology is, however, quite weak. For instance, the function
which associates to a trajectory $\omega \in D([0,T], S\cup \{0\})$
the value $\sup_{0\le t\le T} |\omega(t)|$ is not continuous.  For
this reason, we put forward in this article an alternative definition
of metastability. We propose to declare that the sequence of Markov
chains $\eta^N(t)$ is metastable in the time-scale $\theta_N$ if the
finite-dimensional distributions of $X_N(t) = \Phi_N(\xi^N(t))$
converge to the ones of $\bs X(t)$. Moreover, we show that the
conditions (H1), (H2) together with an extra condition on the visits
to the set $\Delta_N$, stated below in equation \eqref{03}, entail the
metastability of the Markov chains $\eta^N(t)$ in the FDD sense. This
latter result, stated in Proposition \ref{l01} below, is the main
contribution of this article.

We also show, in Proposition \ref{p02}, that conditions (H1), (H2)
together with slightly stronger assumptions entail the convergence of
the state of the process to a time-dependent convex combination of
metastable states.

\section{Notation and Results}
\label{sec-2}

We present in this section the main results of the article. We adopt
the notation introduced in the previous section: $\eta^N(t)$ is an
$E_N$-valued, irreducible Markov chain, whose state space can be
decomposed as in \eqref{17}.

\smallskip\noindent{\bf Convergence of the finite-dimensional
  distributions.} The main result of the article reads as follows:

\begin{proposition}
\label{l01}
Beyond (H1) and (H2), suppose that for all $x\in S$,
\begin{equation}
\label{03}
\lim_{\delta\to 0} \limsup_{N\to\infty} \max_{\eta\in\ms E^x_N} 
\sup_{2\delta\le s\le 3\delta} \bb P_{\eta} [ \xi^N(s) \in \Delta_N ]\;=\;0\;. 
\end{equation}
Then, for all $x\in S$, and all sequences $\{\eta^N : N\ge 1\}$,
$\eta^N\in\ms E^x_N$, under $\bb P_{\eta^N}$ the finite-dimensional
distributions of $X_N(t)$ converge to the finite-dimensional
distributions of the chain $\bs X(t)$.
\end{proposition}

The proof of this result is presented in Section \ref{sec2}, together
with several, easier to verify, sufficient conditions for \eqref{03} to hold.

\smallskip\noindent{\bf Slow variables.} In all models where
metastability has been proved the time-scale $\theta_N$ increases to
$\infty$ with $N$. Since it follows from the previous paragraphs that
the finite-dimensional distributions of $\Phi_N(\xi^N(t))$ converge to
the ones of the Markov chain $\bs X(t)$, we say that $\Phi_N$ is a
\emph{slow variable}.  In this sense, metastability consists in
discovering the slow variables of the system and in deriving their
asymptotic dynamics.

\smallskip\noindent{\bf Convergence of the states.}  We have coined
properties (H1) and (H2) as the metastable behavior of the Markov
chain $\eta^N(t)$ in the time-scale $\theta_N$. However, it has been
pointed out that in mathematical-physics metastability means the
convergence of the state of the process. The second result of this
note fills the gap between these two concepts by establishing that
properties (H1), (H2) together with conditions (M1), (M2) below lead
to the convergence of the state of the process to a convex combination
of states supported on the wells $\ms E^x_N$. The precise statement of
this result requires some notation.

Recall that we denote by $\xi^N(t)$ the Markov chain $\eta^N(t)$
speeded-up by $\theta_N$: $\xi^N(t) = \eta^N(t\theta_N)$.  Denote by
$\mu_N$ the unique stationary state of the chain $\xi^N(t)$, and by
$\mu^y_N$, $y\in S$, the probability measure $\mu_N$ conditioned to
$\ms E^y_N$:
\begin{equation}
\label{203}
\mu^y_N(\xi) \;=\; \frac{\mu_N(\xi)}{\mu_N(\ms E^y_N)}\, 
\mb 1 \{\xi\in \ms E^y_N\}\;, \quad \xi\in E_N \;.
\end{equation}
Note that $\mu^y_N$ is defined on $E_N$ and it is supported on $\ms
E^y_N$.

\smallskip\noindent{\bf Reflected process.} For $x\in S$, denote by
$\{\xi^N_{R,x}(t) : t\ge 0\}$ the Markov chain $\xi^N(t)$ reflected at
$\ms E^x_N$. This is the Markov chain obtained from $\xi^N(t)$ by
forbidding jumps from $\ms E^x_N$ to its complement $(\ms
E^x_N)^c$. This mechanism produces a new Markov chain whose state
space is $\ms E^x_N$, which might be reducible.

We assume that for each $x\in S$ the reflected chain $\xi^N_{R,x}(t)$
is irreducible and that $\mu^x_N$ is its unique stationary state. In
the reversible case this latter assumption follows from the
irreducibility. In the non-reversible case, if the Markov chain
$\eta^N(t)$ is a cycle chain (cf. \cite{Gab12, lseo2016a}) it is easy
to define the sets $\ms E^x_N$ for the reflected chain on $\ms E^x_N$
to be irreducible.  Let $(\mathcal{S}^{R,x}_N(t) : t\ge 0)$, be the semigroup of
the Markov chain $\xi^N_{R,x}(t)$.

\smallskip\noindent{\bf Trace process.}  Similarly, we denote by
$\xi^N_{T,x}(t)$ the trace on $\ms E_N^x$ of the process $\xi^N(t)$,
and by $(\mathcal{S}^{T,x}_N(t) : t\ge 0)$ the semigroup of the Markov chain
$\xi^N_{T,x}(t)$. 

\smallskip\noindent{\bf Mixing times.}  Denote by $\Vert \mu - \nu
\Vert_{\rm TV}$ the total variation distance between two probability measures
defined on the same denumerable set $\Omega$:
\begin{equation*}
\Vert \mu - \nu \Vert_{\rm TV} \;=\; \frac 12\, \sum_{\eta\in \Omega}
\vert \mu (\eta) - \nu (\eta) \vert \;=\; \sum_{\eta\in \Omega}
\big(\mu (\eta) - \nu (\eta)\big)^+\;  , 
\end{equation*}
where $x^+=\max\{x,0\}$ denotes the positive part of $x\in\bb
R$. Hereafter, the set $\Omega$ will be either the set $E_N$ or one of
the wells $\ms E^x_N$, $x\in S$.

Denote by $T^{N,R,x}_{\rm mix}$, $T^{N,T,x}_{\rm mix}$ the
$(\frac{1}{2e})$-mixing time of the reflected, trace processes,
respectively:
\begin{equation}
\label{204}
\begin{gathered}
T^{N,R,x}_{\rm mix} \;=\; \inf \Big\{ t> 0 : \max_{\eta\in \ms E^x_N} 
\Vert \delta_\eta \mathcal{S}^{R,x}_N(t) - \mu^x_N \Vert_{\rm TV} \le \frac
1{2e} \Big\}\;, \\
T^{N,T,x}_{\rm mix} \;=\;  \inf \Big\{ t> 0 : \max_{\eta\in \ms E^x_N} 
\Vert \delta_\eta \mathcal{S}^{T,x}_N(t) - \mu^x_N \Vert_{\rm TV} \le \frac
1{2e} \Big\}\;,
\end{gathered}  
\end{equation}
where $\delta_\eta$ stands for the Dirac measure concentrated on the
configuration $\eta$.

\smallskip\noindent{\bf Hitting times.}  For a subset $\ms A$ of
$E_N$, denote by $H_{\ms A}$, $H^+_{\ms A}$ the hitting time and the
time of the first return to ${\ms A}$:
\begin{equation}
\label{201}
H_{\ms A} \;=\; \inf \big \{t>0 : \xi^N(t) \in {\ms A} \big\}\;,
\quad
H^+_{\ms A} \;=\; \inf \big \{t>\tau_1 : \xi^N(t) \in {\ms A} \big\}\; ,  
\end{equation}
where $\tau_1$ represents the time of the first jump of the chain
$\xi^N(t)$: $\tau_1 = \inf\{t>0 : \xi^N(t) \not = \xi^N(0)\}$. 

For two subsets $\ms A \subset \ms B \subset E_N$, denote by
$H^{\ms B}_{\ms A}$ the hitting time on the set $\ms A$ of the trace
process on $\ms B$:
\begin{equation}
\label{202}
H^{\ms B}_{\ms A} \;=\; \int_0^{H_{\ms A}} \mb 1\{\xi^N(s) \in \ms B\}
\, ds\;.
\end{equation}

Let $(\alpha_N : N\ge 1)$, $(\beta_N : N\ge 1)$ be two sequences of
positive numbers. The relation $\alpha_N \ll \beta_N$ means that
$\lim_{N\to\infty} \alpha_N/\beta_N = 0$. In the next result, we
assume that for each $x\in S$ there exists a set $\ms B_N^x\subset \ms
E_N^x$ fulfilling the following conditions:

\begin{enumerate}
\item[({\bf M1})] For every $\delta>0$ we have
\begin{equation}
\label{b02}
\lim_{N\to\infty} \max_{x\in S} \, \sup_{\eta\in \ms E^x_N} 
\bb P_\eta \big[ H_{\ms B_N^x}^{\ms E_N^x}>\delta\big] \;=\; 0\;.
\end{equation}

\item[({\bf M2})] There exists a time-scale $(\varepsilon_N : N\ge 1)$
  such that $\varepsilon_N\ll 1$,
\begin{equation}
\label{b04}
\lim_{N\to\infty} \max_{x\in S} 
\,\sup_{\eta\in\ms B_N^x}
\bb P_\eta \big[ H_{\Delta_N} \le 2\, \varepsilon_N \big] \;=\; 0\;.
\end{equation}
and
\begin{equation}
\label{b07}
\lim_{N\to\infty} \max_{x\in S} \sup_{\eta\in\ms B_N^x}
\Vert \delta_{\eta} \mathcal{S}^{R,y}_N(\varepsilon_N) - \mu^y_N \Vert_{\rm TV} 
\;=\; 0\;.
\end{equation}
\end{enumerate}

Condition (M1) requires that the process restricted in $\ms E^x_N$ reaches
the set $\ms B^x_N$ quickly. Additionally, condition (M2) imposes that
it takes longer to leave the set $\ms E^x_N$, when starting from $\ms
B^x_N$, than it takes to mix in $\ms E^x_N$. Slightly more precisely,
condition (M2) requests the existence of a time scale $\varepsilon_N$,
longer than the mixing time of the reflected process and shorter than
the exit time from the set $\ms E^x_N$. Note, however, that in
condition \eqref{b07} the initial configuration belongs to the set
$\ms B^x_N$, while in the definition of the mixing time the initial
configuration may be any element of the set $\ms E^x_N$. In any case,
condition \eqref{b07} is in force if $\varepsilon_N \gg T^{N,R,x}_{\rm
  mix}$.

Assume that the chain is reversible. Fix $y\in S$, denote by
$p^{R,y}_t(\zeta,\xi)$ the transition probabilities of the reflected
process $\xi^N_{R,y}(t)$, and fix $\eta\in \ms B^y_N$.  By definition,
\begin{equation*}
\Vert \delta_{\eta} \mathcal{S}^{R,y}_N(\varepsilon_N) - \mu^y_N 
\Vert_{\rm  TV}  \;=\; \frac 12\,
\sum_{\zeta\in \ms E^y_N} |\, f_t (\zeta) - 1\,| \, \mu^y_N (\zeta)\;,
\end{equation*}
where $f_t(\zeta) = p^{R,y}_t(\eta, \zeta)/\mu^y_N (\zeta)$ and $t=
\varepsilon_N$. By Schwarz inequality and a decomposition of $f_t$
along the eigenfunctions of the generator of the reflected process
(cf. equation (12.5) in \cite{lpw}), the square of the previous
expression is bounded by $\exp\{ -2\lambda_{R,y} t\} \Vert f_0
\Vert_{\mu^y_N}^2$, where $\lambda_{R,y}$ represents the spectral gap
of $\xi^N_{R,y}(t)$ and $\Vert f_0 \Vert_{\mu^y_N}$ the norm of $f_0$
in $L^2(\mu^y_N)$. Since 
\begin{equation*}
\Vert f_0 \Vert^2_{\mu^y_N} \;=\; \sum_{\zeta\in \ms E^y_N}
f_0(\zeta)^2\, \mu^y_N (\zeta) \;=\; \sum_{\zeta\in \ms E^y_N}
\frac{\delta_{\eta,\zeta}}{\mu^y_N (\zeta)^2} \, \mu^y_N (\zeta) 
\;=\; \frac 1{\mu^y_N(\eta)}\;,  
\end{equation*}
as $t= \varepsilon_N$, we conclude that
\begin{equation*}
\Vert \delta_{\eta} \mathcal{S}^{R,y}_N(\varepsilon_N) - \mu^y_N 
\Vert_{\rm  TV}  \;\le \; \frac 1{\mu^y_N(\eta)^{1/2}} \, 
e^{ - \lambda_{R,y} \varepsilon_N}\;.
\end{equation*}
Therefore, in the reversible case, condition \eqref{b07} of (M2) is
fulfilled provided
\begin{equation}
\label{l3-07}
\lim_{N\to\infty} \max_{y\in S} \sup_{\eta\in\ms B_N^y}
\frac 1{\mu^y_N(\eta)^{1/2}} \, 
e^{ - \lambda_{R,y} \varepsilon_N} \;=\; 0\;.
\end{equation}

\begin{proposition}
\label{p02}
Assume that conditions (H1), (H2), and (M1), (M2) are in
force. Suppose, furthermore, that for all $y\in S$
\begin{equation}
\label{b09a}
\lim_{N\to\infty} \frac{\mu_N(\Delta_N)}{\mu_N(\ms E^y_N)} \;=\; 0\;,
\end{equation}
and that either of the following three conditions {\em (a), (b) or (c)}
hold.
\begin{enumerate}
\item[(a)] The process $\{\eta^N(t): t\ge 0\}$ is reversible.
\item[(b)] There exists a constant $0<C_0<\infty$, such that for
all $y,z\in S$, $N\ge 1$,
\begin{equation}
\label{b09}
\frac{1}{C_0} \;\le\; \frac{\mu_N(\ms E^z_N)}{\mu_N(\ms E^y_N)} \;\le\; C_0\;.
\end{equation}
\item[(c)] The sets $\ms B_N^x$ referred to in {(M1)} and {(M2)}
further satisfy
\begin{equation}
\label{tracemixing3}
\lim_{N\to\infty} \frac{T^{N,T,x}_{\rm mix}}{\varepsilon_N}\, 
\frac 1{\mu_N^x(\ms B_N^x)} \, 
\Big(1+\ln\big(\frac{1}{\mu_N^x(\ms B_N^x)}\big)\Big) \;=\; 0\;. 
\end{equation}
\end{enumerate} 
For $k\in\mathbb{N}$, $t_1,\ldots,t_k>0$, let $\mu^{N,\eta^N}_{t_1,\ldots,t_k}$ stand for the joint law of $\big(\xi^N(t_1),\ldots,\xi^N(t_k)\big)$, 
with $\xi^N(0)=\eta^N$. Then, for every $x \in S$ and every sequence $\{\eta^N:N\ge
1\}$, $\eta^N \in \ms E^x_N$,
\begin{equation*}
\lim_{N\to\infty} \big\|\mu^{N,\eta^N}_{t_1,\ldots,t_k}-\sum_{y_1,\ldots, y_k\in S}\!\!\bs P_x\big[\bs X(t_1)=y_1,\ldots,\bs X(t_k)=y_k\big]\, \mu^{y_1}_N\times\cdots\times\mu^{y_k}_N\, \big\|_{\rm TV}\;=\;0\;.
\end{equation*}
\end{proposition}

\begin{remark}[On the hypotheses of Proposition \ref{p02}]
\label{rmf2}
Separation of scales, in the sense
that the process mixes in a well before jumping, is a
common feature in metastable Markov chains and it is usually hidden in
the proof of (H1). Conditions (M1)--(M2) is a mathematically
concrete way to elicit this fact in generality.  In the
proof of the metastability of zero-range processes in \cite{agl},
(M1)--(M2) are actually the way (H1) is established. 
On the other
hand, the "either of three" conditions are not so hard to check. This
is clear for reversibility. Condition (b) can be readily checked when $\mu_N$ is known. Moreover, it is always satisfied if the
rates of the limiting process are the rescaled rates of jumping
between wells (which is an assumption for H1) and the limiting Markov
chain is irreducible.  As for (c) there are standard tools to estimate
mixing times (cf. \cite{lpw} in a general set-up, \cite{agl} in the context of metastability and Remark \ref{rm02} below).
\end{remark}

The article is organized as follows. Propositions \ref{l01} and
\ref{p02} are proved in Sections \ref{sec2}, \ref{sec03},
respectively. In Section \ref{sec02} we show that the assumptions of
these propositions are in force for four different classes of
dynamics. In the last section, we present a general bound for the
probability that a hitting time of some set is smaller than a value in
terms of capacities (which can be evaluated by the Dirichlet and the
Thomson principles). Throughout this article, $c_0$ and $C_0$ are
finite positive constants, independent of $N$, whose values may change
from line to line.

\section{Convergence of the finite-dimensional distributions}
\label{sec2}

In this section, we prove Proposition \ref{l01}, and we present some
sufficient conditions for \eqref{03}. We will use the shorthand
\(T_N(t)\) for the time \(T_{\ms E_N}(t)\) spent by the process
\(\xi^N(t)\) in \(\ms E_N\) before time t. Likewise, we will denote
the generalized inverse of \(T_N(t)\) by \(S_N(t)\). Note that
condition (H2) can be stated as
\begin{equation}
\label{02b}
\lim_{N\to\infty}
\max_{\eta\in\ms E_N} \bb E_{\eta} \big[ t-T_N(t) \big] \;=\; 0\;.
\end{equation}
Since $\{S_N(t) \ge t +\delta\} = \{T_N(t +\delta) \le t\} =
\{t+\delta - T_N(t +\delta) \ge \delta\}$, it follows from the
previous equation that for all $t\ge 0$, $\delta>0$,
\begin{equation}
\label{02}
\lim_{N\to\infty}
\max_{\eta\in\ms E_N} \bb P_{\eta} \big[ S_N(t) \ge t +\delta \big]
\;=\; 0\;. 
\end{equation}

The proof of Proposition \ref{l01} is based on the next technical
result, which provides an estimate for the distribution of the trace
process $X^T_N$ in terms of the distribution of the process $X_N$.

\begin{lemma}
\label{a01}
Assume conditions (H1) and (H2). Then, for all $N\ge 1$, $\delta>0$,
$y\in S$, $\eta\in\ms E_N$, and $r>3\delta$,
\begin{equation*}
\bb P_\eta[ X^T_N(r-3\delta) = y ] \;\le\; \bb P_\eta[ X_N(r) = y ]
\;+\; R^{(1)}_N(y, r, \delta) \;+\; R^{(2)}_N(y,\delta) \;,
\end{equation*} % R0 = R1 + R3
where 
\begin{equation*}
\lim_{\delta \to 0} \limsup_{N\to\infty}  R^{(1)}_N(y, r, \delta)\;=\; 0
\end{equation*}
for all $r>0$,  $y\in S$ and 
\begin{equation*}
R^{(2)}_N(y,\delta) \;=\;  \max_{\eta\in\ms E^y_N} 
\sup_{2\delta\le s\le 3\delta} \bb P_{\eta} [ \xi^N(s) \in \Delta_N ]\;.
\end{equation*}
\end{lemma}

\begin{proof}
Fix $N\ge 1$, $\delta>0$, $y\in S$, $\eta\in\ms E_N$ and $r>3\delta$.  By
definition of $X^T_N$, and since $S_N(r-3\delta) \ge r-3\delta$,  
\begin{equation*}
\bb P_\eta[ X^T_N(r-3\delta) = y ] \; =\; 
\bb P_\eta[ X_N(S_N(r-3\delta)) = y ] \;\le\;
\bb P_\eta[ A_N(r,\delta,y) ] \; +\; J^{(1)}_N(\eta, r, \delta)\;, 
\end{equation*}
where
\begin{equation*}
A_N(r,\delta,y) \;=\; \big\{ X_N(s) = y  \text{ for some } r-3\delta\le
s\le r-2\delta \big \}\;, 
\end{equation*}
and
\begin{equation*}
J^{(1)}_N(\eta, r, \delta) \;=\; \bb P_{\eta} [ S_N(r-3\delta) \ge
r-2\delta ]  \;.
\end{equation*}
By \eqref{02} with $t=r-3\delta$,
\begin{equation}
\label{205}
\lim_{N\to\infty} \max_{\eta\in \ms E_N} J^{(1)}_N(\eta, r, \delta)
\;=\; 0\;.
\end{equation}

On the other hand, 
\begin{equation*}
\bb P_\eta[ A_N(r,\delta,y) ] 
\;\le\; \bb P_\eta[ X_N(r) = y ] 
\; +\; \bb P_\eta \big[ A_N(r,\delta,y)
\,,\, \xi^N(r) \not\in \ms E^y_N  \big] \;.
\end{equation*}
Denote by $H$ the first time the process $X_N(s)$ hits the point $y$
after $r-3\delta$:
\begin{equation*}
H \;=\; \inf\{s\ge r-3\delta : X_N(s) =y\}\;.
\end{equation*}
By the strong Markov property, the second term on the right hand side
of the penultimate equation is equal to
\begin{equation*}
\bb E_\eta \Big[  \mb 1\{ H \le r-2\delta\} \, \bb P_{\xi^N(H)} [ 
\xi^N(r-H) \not\in \ms E^y_N]\,  \Big]
\;\le\; \max_{\eta\in\ms E^y_N} \sup_{2\delta\le s\le 3\delta} 
\bb P_{\eta} [ \xi^N(s) \not\in \ms E^y_N]\;.
\end{equation*}

Recall from \eqref{17} the definition of $\breve{\ms E}^y_N$. The
previous probability is bounded by
\begin{equation*}
\bb P_{\eta} [ \xi^N(s) \in \breve{\ms E}^y_N] \;+\;
\bb P_{\eta} [ \xi^N(s) \in \Delta_N] \;.
\end{equation*}
Since $s\le 3\delta$, the first term is bounded by 
\begin{equation*}
J^{(2)}_N (y,\delta) \;:=\;
\max_{\eta\in\ms E^y_N}
\bb P_{\eta} \big [ X^T_N(s') \not = y \text{ for some } s'\le 3
\delta \big]\;.
\end{equation*}
By condition (H1), $J^{(2)}_N (y,\delta)$ vanishes as $N\to \infty$
and then $\delta\to 0$.  To complete the proof of the lemma, it
remains to set $R^{(1)}_N(y, r, \delta) = \max_{\eta\in \ms E_N}
J^{(1)}_N(\eta, r, \delta) + J^{(2)}_N (y,\delta)$ and to recall the
estimate \eqref{205}.
\end{proof}

Denote by $\bs P_x$, $x\in S$, the probability measure on $D(\bb R_+,
S)$ induced by the Markov chain $\bs X(t)$ starting from $x$. Since
$\bs P_x[\bs X(t)\not = \bs X(t-)]=0$ for all $t\ge 0$, the finite-dimensional
distributions of $X^T_N$ converge to the ones of $\bs X(t)$.

\begin{proof}[{\bf Proof of Proposition  \ref{l01}}]
We prove the result for one-dimensional distributions. The extension
to higher order is immediate.  Fix $x$, $y\in S$, $r>0$, and a
sequence $\{\eta^N : N\ge 1\}$, $\eta^N\in\ms E^x_N$.  By assumption
(H1), by Lemma \ref{a01}, and by \eqref{03},
\begin{equation*}
\bs P_x[ \bs X(r) = y] \;=\; \lim_{\delta\to 0}
\bs P_x[ \bs X(r-3\delta) = y] \;\le\; \liminf_{N\to\infty} 
\bb P_{\eta^N}[ X_N(r) = y ]\;.
\end{equation*}
Since
\begin{equation*}
1 \;=\; \sum_{y\in S} \bs P_x[ \bs X(r) = y] \quad\text{and}\quad
\sum_{y\in S} \bb P_{\eta^N}[ X_N(r) = y ] \;\le\; 1\;,
\end{equation*}
the inequality in the penultimate formula must be an identity for each
$y\in S$, which completes the proof of the proposition.
\end{proof}

\subsection{The assumption \eqref{03}}
We conclude this section presenting three sets of sufficient
conditions for the bound \eqref{03}.

\begin{remark}
\label{rm1}
To prove condition \eqref{03}, one is tempted to argue that for
all $2\delta\le s\le 3\delta$, $\eta\in\ms E_N$,
\begin{equation*}
\bb P_{\eta} [ \xi^N(s) \in \Delta_N ] \;\le\;
\bb P_{\eta} [ H_{\Delta_N} \le 3\delta ]\;.
\end{equation*}
In many examples, however, it is not true that the right hand side
vanishes, uniformly over configurations in $\ms E_N$, as $N\to\infty$
and then $\delta\to 0$. In condensing zero ranges processes or in
random walks in a potential field, starting from certain configuration
in a valley $\ms E^x_N$, in a time interval $[0,\delta]$, the chain
$\xi^N(s)$ visits many times the set $\Delta_N$ and the right hand
side of the previous inequality, for such configurations $\eta$, is
close to $1$.
\end{remark}

The next lemma provides sufficient conditions for assumption
\eqref{03} to hold. It is tailor-made to cover the case where the
metastable sets are singletons. This includes spin models on finite
sets \cite{nevsch, sch1, bl2011, bl2015a, CNS2015, ll, lx2015,
  CNS2016}, inclusion processes \cite{bdg,grv}, and random walks among
random traps \cite{jlt1, jlt2}.

\begin{lemma}
\label{l08}
Assume that for each $x\in S$, 
\begin{equation*}
\lim_{N\to\infty} \max_{\eta\in\ms E^x_N} 
\frac {\mu_N(\Delta_N)}{\mu_N(\eta)} \;=\; 0\;.
\end{equation*}
Then, \eqref{03} holds.
\end{lemma}

\begin{proof}
Fix $x\in S$, $\eta\in \ms E^x_N$ and $s>0$.  Multiplying and dividing the
probability $\bb P_{\eta} [ \xi^N(s) \in \Delta_N ]$ by $\mu_N(\eta)$,
we obtain that
\begin{equation*}
\bb P_{\eta} [ \xi^N(s) \in \Delta_N ] \;\le\; 
\frac 1{\mu_N(\eta)} \bb P_{\mu_N} [ \xi^N(s) \in \Delta_N ]
\;=\; \frac {\mu_N(\Delta_N)}{\mu_N(\eta)}\;\cdot
\end{equation*}
In particular, condition \eqref{03} follows from the assumption of the
lemma. 
\end{proof}

The next condition is satisfied by random walks in a potential field
\cite{begk01, lmt2015, lseo2016a, lseo2016b}, illustrated by Example
\ref{ex2}. It is instructive to think of the sets $\ms B^x_N\subset\ms
E^x_N$ below, as the deep part of the well $\ms E^x_N$. The first
condition requires the process to reach the set $\ms B_N^x$ quickly,
while the second one imposes that it will not attain the set
$\Delta_N$ in a short time interval when starting from $\ms B_N^x$.

\begin{lemma}
\label{l04}
Assume that for each $x\in S$ there exists a set $\ms B_N^x\subset \ms
E_N^x$ such that
\begin{gather*}
\lim_{N\to\infty} \sup_{\eta\in \ms E^x_N} 
\bb P_\eta [ H_{\ms B_N^x}>\delta]\;=\;0 \quad\text{for all $\delta>0$}\;, \\
\lim_{\delta \to 0} 
\limsup_{N\to\infty} \sup_{\delta \le s' \le 3\delta}\sup_{\eta\in\ms B_N^x}
\bb P_\eta [ \xi^N(s') \in \Delta_N ] \;=\; 0\;.
\end{gather*}
Then, condition \eqref{03} is in force.
\end{lemma}

\begin{proof}
Fix $x\in S$, $\eta \in \ms E^x_N$, $\delta>0$, $s\in
[2\delta, 3\delta]$, and write
\begin{equation*}
\bb P_{\eta} [ \xi^N(s) \in \Delta_N ] \;\le\;
\bb P_{\eta} [ H_{\ms B_N^x} \le \delta \,,\, \xi^N(s) \in \Delta_N ]
\;+\; \bb P_{\eta} [ H_{\ms B_N^x} > \delta ]\;.
\end{equation*}
On the event $\{H_{\ms B_N^x}<+\infty\}$ let us denote $\xi_{\ms
  B}^N=\xi^N(H_{\ms B_N^x})$. By the strong Markov property and since
$s$ belongs to the interval $[2\delta, 3\delta]$, the first term on
the right hand side is bounded by
\begin{equation*}
\bb E_{\eta} \Big[  \mb 1\{H_{\ms B_N^x} \le \delta\} \,  
\bb P_{\xi^N_{\ms B}} [\xi^N(s-H_{\ms B_N^x}) \in \Delta_N ] \, \Big]
\;\le\; \sup_{\delta \le s' \le 3\delta}\sup_{\eta\in\ms B_N^x}
\bb P_\eta [ \xi^N(s') \in \Delta_N ]\;,
\end{equation*}
which completes the proof of the lemma.
\end{proof}

In Lemmata \ref{ltr} and \ref{l06a} below we present some conditions which imply that,
for all $\delta>0$, $\sup_{\eta\in \ms E^x_N} \bb P_{\eta} [H_{\ms
  B^x_N} \ge \delta]$ vanishes as $N\to\infty$.

Recall from \eqref{203}, \eqref{204} that $\mu^x_N$ represents the
stationary measure $\mu_N$ conditioned to $\ms E^x_N$, and
$\mathcal{S}^{R,x}_N(t)$ the semigroup of the reflected process on $\ms
E^x_N$. The third set of conditions which yield \eqref{03} relies on
the next estimate.

\begin{lemma}
\label{l06}
For every $0<T<\delta <s$, $x\in S$, and configuration $\eta\in \ms
E^x_N$, 
\begin{equation*}
\bb P_{\eta} [ \xi^N(s) \in \Delta_N ] \;\le\;
\bb P_{\eta} [ H_{(\ms E^x_N)^c} \le T ] \;+\; 
\Vert \delta_{\eta} \mathcal{S}^{R,x}_N (T) - \mu^x_N \Vert_{\rm TV} \;+\;
\frac {\mu_N(\Delta_N)}{\mu_N(\ms E^x_N)}\, \;. 
\end{equation*}
\end{lemma}

\begin{proof}
Fix $x\in S$, $\eta \in \ms E^x_N$, and $0<T<\delta <s$. Clearly, 
\begin{equation*}
\bb P_{\eta} [ \xi^N(s) \in \Delta_N ] \;\le\;
\bb P_{\eta} [ H_{(\ms E^x_N)^c} \le T ] \;+\; 
\bb P_{\eta} [ \xi^N(s) \in \Delta_N \,,\, H_{(\ms E^x_N)^c} > T]\;.
\end{equation*}
On the set $\{H_{(\ms E^x_N)^c} > T\}$, up to time $T$, we may couple
the chain $\xi^N(s)$ with the chain reflected at the boundary of $\ms
E^x_N$, which has been denoted by $\xi^N_{R,x}(s)$. By the Markov
property at time $T$ and replacing $\xi^N(s)$ by $\xi^N_{R,x}(s)$, the
second term of the previous equation becomes
\begin{align*}
& \sum_{\zeta\in\ms E^x_N} \bb P_{\eta} \big[ \xi^N_{R,x} (T) = \zeta \,,\,
H_{(\ms E^x_N)^c} > T \big] \, \bb P_{\zeta} [\xi^N(s-T) \in \Delta_N ] \\
&\quad \;\le\;
\sum_{\zeta\in\ms E^x_N} \bb P_{\eta} \big[ \xi^N_{R,x} (T) =
\zeta \big] \, \bb P_{\zeta} [\xi^N(s-T) \in \Delta_N ] \;.  
\end{align*}
By definition of the total variation distance, and since, by
assumption, the stationary measure of the reflected process is given
by $\mu^x_N = \mu_N/\mu_N(\ms E_N^x)$, this sum is less than or equal
to
\begin{equation*}
\Vert \delta_{\eta} \mathcal{S}^{R,x}_N (T) - \mu^x_N \Vert_{\rm TV} \;+\;
\frac 1{\mu_N(\ms E_N^x)}\, \bb P_{\mu_N} [\xi^N(s-T) \in \Delta_N
]\;. 
\end{equation*}
The second term is equal to $\mu_N(\Delta_N)/\mu_N(\ms E_N^x)$, which
completes the proof of the lemma.
\end{proof}
Recall from \eqref{202} that we denote by $H_{\ms B_N^x}^{\ms E_N}$
($H_{\ms B_N^x}^{\ms E_N^x}$, respectively) the hitting time on the
set $\ms B_N^x$ for the trace process on $\ms E_N$ ($\ms E_N^x$,
respectively).
\begin{lemma}\label{ltr}
Under assumptions (H1) and (H2), for every $x\in S$ and $\eta\in\ms E_N^x$,
\begin{equation}
\lim_{N\to\infty}\bb P_\eta \big[ H_{\ms B_N^x}^{\ms E_N^x}>\delta\big] =0,\ \forall \delta>0\ \Longleftrightarrow\
\lim_{N\to\infty}\bb P_\eta \big[ H_{\ms B_N^x}>\delta \, \big]=0,\ \forall \delta>0.
\end{equation}
\end{lemma}
\begin{proof}
By definition of $H_{\ms B_N^x}^{\ms E_N}$, for any
$\eta\in\ms E_N^x\setminus\ms B_N^x$,
\begin{align*}
& \bb P_\eta \big[ H_{\ms B_N^x}>\delta \, \big] 
\;\le\; \bb P_\eta \big[ H_{\ms B_N^x}^{\ms E_N}>T_N(\delta)\big] \\
&\qquad \le \; \bb P_\eta \big[ T_N(\delta) \;\le\;
\frac{\delta}{2} \, \big] \;+\; 
\bb P_\eta \big[ H_{\ms B_N^x}^{\ms E_N}>T_N(\delta),\ T_N(\delta)> 
\frac{\delta}{2} \, \big]\;.
\end{align*}
The first term on the right hand side of the preceding equation
vanishes ,as $N\to\infty$, by \eqref{02b}. The second term is bounded by
\begin{align*}
\bb P_\eta \big[ H_{\ms B_N^x}^{\ms E_N}>\frac{\delta}{2} \, \big]
\; \le\; \bb P_\eta \big[ H_{\ms E_N\setminus \ms E_N^x}^{\ms E_N}
\; \le\; \frac{\delta}{2} \, \big] \;+\;
\bb P_\eta \big[ H_{(\ms E_N^x\setminus \ms B_N^x)^c}^{\ms E_N}
>\frac{\delta}{2}\, \big]\;.
\end{align*}
Since the event $\{H_{\ms E_N\setminus \ms E_N^x}^{\ms E_N} \le
\delta/2\}$ can be expressed as $\{X^T_N(s) \not = X^T_N(0)$ for some
$0<s\le \delta/2\}$, by assumption (H1), the first term on the right
hand side of the last equation vanishes as $N\to\infty$.  Just as in
the proof of Lemma \ref{l06}, on the event $\{H_{\ms E_N\setminus \ms
  E_N^x}^{\ms E_N}> \delta/2\}$ we may couple the trace process
$\xi^{\ms E_N}$ with the trace process on the well $\ms E_N^x$. This
permits to bound the last term in the preceding equation by $\bb
P_\eta [ H_{\ms B_N^x}^{\ms E_N^x}> \delta/2]$. Hence,
\begin{equation}
\label{tracewilldo}
\limsup_{N\to\infty} \sup_{\eta\in \ms E^x_N} 
\bb P_\eta \big[ H_{\ms  B_N^x}>\delta \big] \;\le\; 
\limsup_{N\to\infty} \sup_{\eta\in \ms E^x_N} 
\bb P_\eta \big[ H_{\ms B_N^x}^{\ms E_N^x}>\frac{\delta}{2} \big]\; .
\end{equation}
The reverse implication is trivial, since $H_{\ms B_N^x}\ge
H_{\ms B_N^x}^{\ms E_N^x},$ pointwise.
\end{proof}

\begin{corollary}
\label{l07}
Assume that conditions (H1), (H2), (M1), (M2) and \eqref{b09a} are in force. 
Then, condition \eqref{03} is satisfied. In particular, under the assumptions of Proposition \ref{p02}, the finite-dimensional
distributions of the projected chain $X_N(t)$ converge, as $N\to\infty$,
to the finite-dimensional distributions of the Markov chain $\bs X(t)$
appearing in condition (H1).
\end{corollary}

\begin{proof}[Proof of Corollary \ref{l07}]
The first assertion of the corollary is a straightforward consequence of the
assumptions and Lemma \ref{l04}, Lemma \ref{l06}, Lemma \ref{ltr} with $\eta\in\ms
B_N^x$. The second assertion follows from Proposition \ref{l01}.
\end{proof}

Denote by $\lambda_N(\eta)$, $\eta\in E_N$, the holding rates of the
Markov chain $\xi^N(t)$.  For two disjoint subsets $\ms A$, $\ms B$ of
$E_N$, denote by $\Cap_N(\ms A, \ms B)$ the capacity between $\ms A$ and
$\ms B$:
\begin{equation}
\label{l3-19}
\Cap_N(\ms A, \ms B) \;=\; \sum_{\eta\in \ms A} \mu_N(\eta)\, \lambda_N(\eta) 
\, \bb P_{\eta} [H_{\ms B} < H^+_{\ms A}]\;.
\end{equation}
Similarly, for two disjoint subsets $\ms A$, $\ms B$ of $\ms E^x_N$ we
represent by $\Cap_{N,x}(\ms A, \ms B)$ the capacity between $\ms A$
and $\ms B$ for the trace process $\xi^N_{T,x}(t)$:
\begin{equation*}
\Cap_{N,x}(\ms A, \ms B) \;=\; \sum_{\eta\in \ms A} \mu^x_N(\eta)\, \lambda^{T,x}_N(\eta) 
\, \bb P_{\eta} [H_{\ms B} < H^+_{\ms A}]\;,
\end{equation*}
where $\lambda^{T,x}_N(\eta)$ stands for the holding rates of the trace
process $\xi^N_{T,x}(t)$.

The following lemma offers sufficient conditions for (M1), in terms of mixing time or capacity estimates. 
In view of Lemma \ref{ltr}, together with (H1) and (H2) these conditions also imply that
\begin{equation}
\label{hitbottom}
\lim_{N\to\infty} \sup_{\eta\in \ms E^x_N}
\bb P_\eta [ H_{\ms B_N^x}>\delta]=0,\quad \forall \delta>0.
\end{equation}
\begin{lemma}
\label{l06a}
Let $T^{N,T,x}_{\rm mix}$ represent the
$\big(\frac{1}{2e}\big)$-mixing time of the trace process on $\ms
E_N^x$. If, for every $x\in S$ either
\begin{equation}
\label{capest}
\lim_{N\to\infty} \sup_{\eta\in \ms E_N^x\setminus \ms B_N^x} 
\frac{\mu_N(\ms E_N^x \setminus \ms B_N^x)}
{\Cap_N(\eta,\ms B_N^x)} \;=\; 0 \;,
\end{equation}
or
\begin{equation}
\label{tracemixing2}
\lim_{N\to\infty} \frac{T^{N,T,x}_{\rm mix}}{\mu_N^x(\ms B_N^x)}
\, \Big(1+\ln\big(\frac{1}{\mu_N^x(\ms B_N^x)}\big)\Big) \;=\; 0\,
\end{equation}
are satisfied, then (M1) holds. If $\{\eta^N(t): t\ge
0\}$ is reversible, the logarithmic term in \eqref{tracemixing2} can
be dropped.
\end{lemma}

\begin{proof}

To prove the assertion of the lemma under the assumption
\eqref{capest}, note that by Proposition A.2 in \cite{bl4},
\begin{align*}
\bb E_{\eta} \big[ H_{\ms B_N^x}^{\ms E_N^x} \big] \;\le\; 
\frac{\mu_N^x(\ms E_N^x\setminus \ms B_N^x)}{\Cap_{N,x}(\eta,\ms B_N^x)}
\;=\; \frac{\mu_N(\ms E_N^x\setminus\ms B_N^x)}{\mu_N(\ms E_N^x)
\Cap_{N,x}(\eta,\ms B_N^x)}
\;=\; \frac{\mu_N(\ms E_N^x\setminus\ms B_N^x)}{\Cap_N(\eta,\ms B_N^x)}\;,
\end{align*}
where the last equality follows from identity (A.10) in \cite{bl4}.

Assume, now, that \eqref{tracemixing2} is in force. The following
argument is inspired by Theorem 6 in \cite{aldous} and Theorem 1.1 in
\cite{peressousi}. We include it here for completeness.  Recall from
\eqref{204} that we denote by $\mathcal{S}^{T,x}_N(t)$ the semigroup of the trace
process on $\ms E^x_N$. Pick a time-scale $(\vartheta_N : N\ge 1)$
such that
\begin{equation*}
\sup_{\eta\in\ms E_N^x}\Vert \delta_{\eta}\mathcal{S}^{T,x}_N(\vartheta_N)-
\mu_N^x\Vert_{\rm TV} < \mu_N^x(\ms B_N^x)/2.
\end{equation*}
We may choose, for example,
\begin{equation}
\label{206}
\vartheta_N \;=\; 
\Big(1+\ln\big(\frac{1}{\mu_N^x(\ms B_N^x)}\big)\Big)\,
T^{N,T,x}_{\rm mix} \;.
\end{equation}
Recall that we denote by $\xi^N_{T,x}(t)$ the trace of the Markov
chain $\xi^N(t)$ on $\ms E^x_N$. For any $\eta\in\ms E_N^x$, by
definition of $\vartheta_N$
\begin{align*}
& \bb P_\eta [ H_{\ms B_N^x}^{\ms E_N^x}>\vartheta_N] \;\le\;
\bb P_\eta [ \xi^N_{T,x} (\vartheta_N)\notin \ms B_N^x] \\
& \qquad
\le\; \Vert \delta_{\eta}\mathcal{S}^{T,x}_N (\vartheta_N)-\mu_N^x\Vert_{\rm TV} 
\;+\; \mu_N^x(\ms E_N^x\setminus\ms B_N^x)
\;\le\; 1 \,-\, \frac{\mu_N^x(\ms B_N^x)}{2}\;\cdot
\end{align*}
Since this estimate is uniform in $\eta$, we may iterate it, using the
Markov property, to get
\begin{equation}
\label{iterated}
\bb P_\eta [ H_{\ms B_N^x}^{\ms E_N^x}>\delta] \; \le\;
\Big(1-\frac{\mu_N^x(\ms B_N^x)}{2}\Big)^{\left[\frac{\delta}{\vartheta_N}\right]}\;.
\end{equation}
This expression vanishes, as $N\to \infty$, if \eqref{tracemixing2} is
satisfied and if we choose $\vartheta_N$ according to \eqref{206}.

Finally, if the process is reversible, by Theorem 5 in \cite{aldous},
there exists a finite universal constant $C_0$ such that
\begin{equation*}
\mu_N^x(\ms B_N^x)\, \bb E_{\eta} \big[ H_{\ms B_N^x}^{\ms
  E_N^x} \big] \;\le\; C_0\, T^{N,T,x}_{\rm mix}\;.
\end{equation*}
Hence, \eqref{hitbottom} follows from (\ref{tracewilldo}) by Markov's
inequality.
\end{proof}

The preceding lemma evidences the importance of an upper bound for the
mixing time of the trace process. This is the content of Remark
\ref{rm02} below.

Denote by $R_N(\eta,\xi)$, $\eta$, $\xi\in E_N$, the jump rates of the
Markov chain $\xi^N(t)$, and by $R^{T,x}_N(\eta',\xi')$, $\eta'$,
$\xi'\in \ms E^x_N$, the jump rates of the trace process
$\xi^N_{T,x}(t)$. Assume that the Markov chain $\xi^N(t)$ is
reversible and denote by $D_N$, $D_{N,T,x}$ the Dirichlet form of the
processes $\xi^N(t)$, $\xi^N_{T,x}(t)$, respectively:
\begin{equation}
\label{l3-04}
\begin{aligned}
& D_N(f) \;=\; \frac 12 \sum_{\eta,\xi\in E_N} \mu_N(\eta)\,
R_N(\eta,\xi)\, [f(\xi) - f(\eta)]^2\;, \\
&\qquad  D_{N,T,x}(g) \;=\; \frac 12 \sum_{\eta,\xi\in \ms E^x_N} \mu^x_N(\eta)\,
R^{T,x}_N (\eta,\xi)\, [g(\xi) - g(\eta)]^2\;,
\end{aligned}
\end{equation}
for functions $f:E_N \to \bb R$, $g:\ms E^x_N \to \bb R$. By replacing,
in the first line of the previous formula, the measure $\mu_N$ by the
conditioned measure $\mu^x_N$, and by restricting the sum to
configurations $\eta,\xi\in \ms E^x_N$, we obtain the Dirichlet form
of the reflected process, denoted by $D_{N,R,x}(f)$.

Denote by $T^{N,T,x}_{\rm rel}$, $T^{N,R,x}_{\rm rel}$ the relaxation
times of the trace process $\xi^N_{T,x}(t)$, the reflected process
$\xi^N_{R,x}(t)$, respectively:  
\begin{equation*}
T^{N,T,x}_{\rm rel} \;=\; \sup_g \frac{\Vert
  g\Vert^2_{\mu^x_N}}{D_{N,T,x}(g)} \;, \quad
T^{N,R,x}_{\rm rel} \;=\; \sup_g \frac{\Vert
  g\Vert^2_{\mu^x_N}}{D_{N,R,x}(g)} \;,
\end{equation*}
where the supremum is carried over all functions $g:\ms E^x_N \to \bb
R$ with mean zero with respect to $\mu^x_N$, and $\Vert
g\Vert_{\mu^x_N}$ represents the $L^2(\mu^x_N)$ norm of $g$: $\Vert
g\Vert^2_{\mu^x_N} = \sum_{\eta\in \ms E^x_N} \mu^x_N(\eta) \,
g(\eta)^2$. 

\begin{remark}
\label{rm02}
Obtaining estimates for the mixing time $T^{N,T,x}_{\rm mix}$ of the
trace process on the well $\ms E_N^x$ is often not harder than doing
so for the mixing time $T_{\rm mix}^{N,R,x}$ of the reflected process
on the well. Both processes have the same invariant measure $\mu_N^x$
and the former has higher jump rates. Indeed, by \cite[Corollary
6.2]{bl2}, for any $\eta,\zeta\in\ms E_N^x$, $\eta\not =\zeta$,
\begin{equation*}
R^{T,x}_N(\eta,\zeta) \;=\; R_N(\eta,\zeta) \;+\;
\sum_{\xi\notin \ms E_N^x} R_N(\eta,\xi) \,
\bb{P}_\xi\big[\xi^N\big(H_{\ms E_N^x}\big) =\zeta\big]
\;\ge\; R_N(\eta,\zeta)\;.
\end{equation*}
Hence, the Dirichlet form corresponding to the trace on $\ms E_N^x$
dominates the Dirichlet form corresponding to the reflected process on
$\ms E_N^x$ and, consequently, the relaxation time $T_{\text{\rm
    rel}}^{N,T,x}$ of the former is smaller than the relaxation
time $T_{\text{\rm rel}}^{R,x}$ of the latter. Then, by the proof of
\cite[Theorem 12.3]{lpw},
\begin{equation}
T^{N,T,x}_{\rm mix} \;\le\; T^{N,T,x}_{\rm rel}  \,
\Big(1+\sup_{\eta\in\ms E_N^x}\log\big(\frac{1}{\mu_N^x(\eta)}\big)\Big)
\;\le\; T^{N,R,x}_{\rm rel} \Big(1+\sup_{\eta\in\ms E_N^x}
\log\big(\frac{1}{\mu_N^x(\eta)}\big)\Big)\;.
\label{tracemixestimate}
\end{equation}
The right hand side of the preceding inequality, which is often used
as an upper bound for the mixing time $T_{\text{mix}}^{N,R,x}$ of the
chain $\xi^N(\cdot)$ restricted in the well $\ms E_N^x$, is also a
bound for the mixing time of the trace process.
\end{remark}

\begin{remark}
\label{rm03}
In many interesting cases, e.g. random walks on a potential field
\cite{begk01, lmt2015, lseo2016a, lseo2016b} or condensing zero-range
processes \cite{bl2012a, l2014}, the set $\ms B_N^x$ may be taken as a
singleton.
\end{remark}

\section{Convergence of the state}
\label{sec03}
In this section, we prove Proposition \ref{p02}. From now on, we
assume that the number of valleys is fixed and that the sequence of
Markov chains fulfills conditions (H1), (H2), (M1), (M2) and \eqref{b09a}.

\begin{proof}[{\bf Proof of Proposition \ref{p02}}] We will prove the assertion for $k=1$.
The general case follows easily from Corollary \ref{l07} and the Markov property.
The proof is divided in several steps. At each stage we write the main
expression as the sum of a simpler one and a negligible remainder.

Fix $t>0$, $x \in S$, a sequence $\{\eta^N: N\ge 1\}$, $\eta^N \in
\ms E^x_N$, and $0<\delta<t$. By definition, $[\delta_{\eta^N}
S^N(t)](\xi)  = \bb P_{\eta^N}[\xi^N(t) = \xi]$ can be written as
\begin{equation}
\label{b01}
\sum_{y\in S} \sum_{\eta\in \ms E^y_N} 
\bb P_{\eta^N}[\xi^N(t-\delta) = \eta \,,\, \xi^N(t) = \xi] \;+\;
R^{(1)}_N(t,\delta,\xi) \;, 
\end{equation}
where
\begin{equation*}
R^{(1)}_N(t,\delta,\xi) \;=\;
\sum_{\eta\in \Delta_N} \bb P_{\eta^N}[\xi^N(t-\delta) = \eta \,,\,
\xi^N(t) = \xi]\;. 
\end{equation*}
By Corollary \ref{l07}, for every $0<\delta<t$,
\begin{equation*}
\lim_{N\to\infty} \sum_{\xi\in E_N} R^{(1)}_N(t,\delta,\xi) \;=\;
\lim_{N\to\infty} \bb P_{\eta^N}[ X_N (t-\delta) = 0] \;=\; 0\;.
\end{equation*}

By the Markov property, the sum appearing in \eqref{b01} is equal to 
\begin{align*}
\sum_{y\in S} \sum_{\eta\in \ms E^y_N} 
\bb P_{\eta^N}[\xi^N(t-\delta) = \eta] \,  
\bb P_{\eta}[\xi^N(\delta) = \xi] \;.
\end{align*}
Let $p(\eta) = \bb P_{\eta^N}[\xi^N(t-\delta) = \eta]$. We may rewrite
this expression as
\begin{equation}
\label{b03}
\sum_{y\in S} \sum_{\eta\in \ms E^y_N} p(\eta)\,
\bb P_{\eta}[H_{\ms B_N^y} \le \frac{\delta}{2} \,,\, \xi^N(\delta) = \xi]
\;+\; R^{(2)}_N(t,\delta,\xi)\;,
\end{equation}
where 
\begin{equation*}
\sum_{\xi\in E_N} R^{(2)}_N(t,\delta,\xi) 
\;\le\; \max_{y\in S} \max_{\eta\in \ms E^y_N}
\bb P_{\eta} \big[  H_{\ms B_N^y} > \frac{\delta}{2} \big].
\end{equation*}
By \eqref{tracewilldo} and condition (M1),
\begin{equation*}
\lim_{N\to\infty} \sum_{\xi\in E_N} R^{(2)}_N(t,\delta,\xi) \;=\;0\;.
\end{equation*}

By the strong Markov property, using the notation $\xi_{\ms
  B}^N=\xi^N(H_{\ms B_N^y})$, the first term in \eqref{b03} is equal
to
\begin{equation*}
\sum_{y\in S} \sum_{\eta\in \ms E^y_N} 
p(\eta) \,  
\bb E_{\eta} \Big[ \bs {1}\{ H_{\ms B_N^y} \le \frac{\delta}{2}\}\,
\bb P_{\xi^N_{\ms B}} [\xi^N(\delta - H_{\ms B_N^y}) = \xi]\, \Big] \;.
\end{equation*}
Let us now define $A_\xi = \{\xi^N(\delta - H_{\ms B_N^y}) = \xi\}$,
$B_N= \{ H_{\ms B_N^y} \le \frac{\delta}{2}\}$ and recall the
definition of the time-scale $\varepsilon_N$ introduced in
condition (M2). Rewrite the previous sum as
\begin{equation}
\label{b05}
\sum_{y\in S} \sum_{\eta\in \ms E^y_N} p(\eta) \,  
\bb E_{\eta} \Big[ \bs {1} \{ B_N \}\,
\bb P_{\xi^N_{\ms B}} [ H_{\Delta_N} > \varepsilon_N \,,\, A_\xi]\, \Big] 
\;+\; R^{(3)}_N(t,\delta,\xi)\;,
\end{equation}
where 
\begin{align*}
0\le\sum_{\xi\in E_N} R^{(3)}_N(t,\delta,\xi) 
\; &=\; \sum_{y\in S} \sum_{\eta\in \ms E^y_N} p(\eta) \,  
\bb E_{\eta} \Big[ \bs {1} \{ B_N \}\,
\bb P_{\xi^N_{\ms B}} [ H_{\Delta_N} \le \varepsilon_N \,]\, \Big]\\
&\le\;  \max_{y\in S} 
\,\sup_{\eta\in\ms B_N^y}
\bb P_\eta [ H_{\Delta_N} \le \varepsilon_N] \;.
\end{align*}
By \eqref{b04}, this latter expression vanishes as $N\to\infty$.

By the Markov property, the sum appearing in \eqref{b05} is equal to
\begin{equation*}
\sum_{y\in S} \sum_{\eta\in \ms E^y_N} p(\eta) \,  
\bb E_{\eta} \Big[ \bs {1} \{B_N\}\,
\bb E_{\xi^N_{\ms B}} \Big[ \bs {1}\{ H_{\Delta_N} > \varepsilon_N \}\, 
\bb P_{\xi^N(\varepsilon_N)} [A'_\xi] \,\Big]\, \Big] \;,
\end{equation*}
where $A'_\xi = \{\xi^N(\delta - H_{\ms B_N^y} - \varepsilon_N) = \xi\}$.
On the set $\{ H_{\Delta_N} > \varepsilon_N \}$, we may replace the chain
$\xi^N(t)$ by the reflected chain at $\ms E^y_N$, denoted by
$\xi^N_{R,y}(t)$. The previous expression is thus equal to
\begin{equation*}
\sum_{y\in S} \sum_{\eta\in \ms E^y_N} p(\eta) \,  
\bb E_{\eta} \Big[ \bs {1} \{B_N\}\,
\bb E_{\xi^N_{\ms B}} \Big[ \mb 1\{ H_{\Delta_N} > \varepsilon_N \}\, 
\bb P_{\xi^N_{R,y}(\varepsilon_N)} [A'_\xi] \,\Big]\, \Big] \;.
\end{equation*}
This sum can be rewritten as
\begin{equation}
\label{b06}
\sum_{y\in S} \sum_{\eta\in \ms E^y_N} p(\eta) \,  
\bb E_{\eta} \Big[ \bs { 1} \{ B_N \}\,
\bb E_{\xi^N_{\ms B}} \Big[\bb P_{\xi^N_{R,y}(\varepsilon_N)} [A'_\xi]
\,\Big]\, \Big] \;-\; R^{(4)}_N(t,\delta,\xi)\;,
\end{equation}
where, by \eqref{b04} and a similar argument to the one following \eqref{b05}
\begin{equation*}
\lim_{N\to\infty} \sum_{\xi\in E_N} R^{(4)}_N(t,\delta,\xi) \;=\;0\;.
\end{equation*}

Since, for every $\eta\in \ms B_N^y,\, \xi\in E_N$,
\begin{equation*}
\bb E_{\eta} \Big[\bb P_{\xi^N_{R,y}(\varepsilon_N)}
[A'_\xi]\,\Big] \;=\; \bb P_{\mu^y_N} [A'_\xi] \;+\; 
\sum_{\zeta\in\ms E^y_N}
\big\{ \bb P_{\eta} \big[ \xi^N_{R,y}(\varepsilon_N) = \zeta \big] 
- \mu^y_N (\zeta) \big\} \, \bb P_{\zeta} [A'_\xi]\;,
\end{equation*}
the first term of \eqref{b06} is equal to
\begin{equation}
\label{b08}
\sum_{y\in S} \sum_{\eta\in \ms E^y_N} p(\eta) \, 
\bb E_{\eta} \Big[ \bs { 1} \{ B_N \}\,
\bb P_{\mu^y_N} [A'_\xi] \,\Big] \;+\; R^{(5)}_N(t,\delta,\xi)\;, 
\end{equation}
where the remainder $R^{(5)}_N(t,\delta,\xi)$ is given by
\begin{equation*}
\sum_{y\in S} \sum_{\eta\in \ms E^y_N}  \, p(\eta) \, 
\bb E_{\eta} \Big[ \bs {1} \{ B_N \} \, \sum_{\zeta\in\ms E^y_N}
\Big\{ \bb P_{\xi^N_{\ms B}} \big[ \xi^N_{R,y}(\varepsilon_N) = \zeta \big] 
- \mu^y_N (\zeta) \Big\}\,
\bb P_{\zeta} [A'_\xi] \,\Big]\;.
\end{equation*}
Therefore,
\begin{equation*}
\sum_{\xi\in E_N} \big| R^{(5)}_N(t,\delta,\xi) \big| \;\le\;
2\,\max_{y\in S}\sup_{\eta\in\ms B_N^y} \Vert \delta_{\eta} S^{R,y}_N(\varepsilon_N) 
- \mu^y_N \Vert_{\rm TV}\;,
\end{equation*}
so that, by \eqref{b07},
\begin{equation*}
\lim_{N\to\infty} \sum_{\xi\in E_N} \big| R^{(5)}_N(t,\delta,\xi)
\big| \;=\; 0\;.
\end{equation*}

The first term in \eqref{b08} can be written as
\begin{equation}
\label{b10}
\sum_{y\in S} \sum_{\eta\in \ms E^y_N} p(\eta) \, 
\bb E_{\eta} \Big[ \bs { 1} \{ B_N \}\,
\bb P_{\mu^y_N} [A'_\xi] \,\Big]  \, \bs{1}\{\xi \in \ms E^y_N\}
\;+\; R^{(6)}_N(t,\delta,\xi)\;,  
\end{equation}
where
\begin{equation*}
R^{(6)}_N(t,\delta,\xi) \;=\;
\sum_{y\in S} \sum_{\eta\in \ms E^y_N} p(\eta) \, 
\bb E_{\eta} \Big[ \bs{1} \{ B_N \}\,
\bb P_{\mu^y_N} [A'_\xi] \,\Big]  \, \bs {1}\{\xi \not\in \ms
E^y_N\} \;.
\end{equation*}
Therefore,
\begin{equation*}
\sum_{\xi\in E_N} R^{(6)}_N(t,\delta,\xi) \;\le\;
\sum_{y\in S} \sum_{\eta\in \ms E^y_N} p(\eta) \, 
\bb E_{\eta} \Big[ \bs {1} \{ B_N \}\,
\bb P_{\mu^y_N} [\xi^N(\delta - H_{{\ms B_N^y}} - \varepsilon_N) \not\in \ms
E^y_N] \,\Big]  \;.
\end{equation*}
The probability inside the expectation is less than or equal to
\begin{equation*}
\bb P_{\mu^y_N} [\xi^N(\delta - H_{\ms B_N^y} - \varepsilon_N) \in \Delta_N] \;+\;
\bb P_{\mu^y_N} [\xi^N(\delta - H_{\ms B_N^y} - \varepsilon_N) \in \breve{\ms
E}^y_N]\;,
\end{equation*}
where $\breve{\ms E}^y_N$ has been introduced in \eqref{17}.  Since
$\mu^y_N (\zeta) = \mu_N(\ms E^y_N)^{-1} \mu_N( \zeta \cap \ms E^y_N)
\le \mu_N(\ms E^y_N)^{-1} \mu_N( \zeta)$, the first term is bounded by
$\mu_N(\Delta)/\mu_N(\ms E^y_N)$. On the other hand, the second term
is less than or equal to
\begin{equation*}
\bb P_{\mu^y_N} [ \sup_{0\le s\le\delta} |X^T_N(s) - y| \ge 1] \;.
\end{equation*}
Therefore,
\begin{equation*}
\sum_{\xi\in E_N} R^{(6)}_N(t,\delta,\xi) \;\le\; 
\max_{y\in S} \Big\{ \frac {\mu_N(\Delta)}{\mu_N(\ms E^y_N)} \;+\;
\bb P_{\mu^y_N} [ \sup_{0\le s\le\delta} 
|X^T_N(s) - y| \ge 1] \Big\} \;,
\end{equation*}
and, by assumption (H1) and \eqref{b09},
\begin{equation*}
\lim_{\delta\to 0} \limsup_{N\to\infty} \sum_{\xi\in E_N}
R^{(6)}_N(t,\delta,\xi) \;=\;0\;.
\end{equation*}

Lemma \ref{3cond} below shows that the first term in \eqref{b10} is
equal to
\begin{equation}
\label{b11n}
\sum_{y\in S} \sum_{\eta\in \ms E^y_N} p(\eta) \, 
\bb P_{\eta} \Big [ H_{\ms B_N^y} \le \frac{\delta}{2} \big]  \, 
\mu^y_N(\xi) \;+\; R^{(7)}_N(t,\delta,x) \;, 
\end{equation} 
where
\begin{equation*}
\lim_{\delta\to 0} \limsup_{N\to\infty} \sum_{\xi\in E_N}
|R^{(7)}_N(t,\delta,x)| \;=\; 0\;.
\end{equation*}

We may rewrite the sum in \eqref{b11n} as
\begin{equation}
\label{b12}
\sum_{y\in S} \sum_{\eta\in \ms E^y_N} p(\eta) \, 
\mu^y_N(\xi) \;-\; R^{(8)}_N(t,\delta,\xi) \;,   
\end{equation}
where
\begin{equation*}
R^{(8)}_N(t,\delta,\xi) \;=\; \sum_{y\in S} \sum_{\eta\in \ms E^y_N} p(\eta) \, 
\bb P_{\eta} \Big [ H_{\ms B_N^y} > \frac{\delta}{2} \Big]  \, \mu^y_N(\xi) \;. 
\end{equation*}
By \eqref{tracewilldo} and condition (M1), for every $0<\delta<t$,
\begin{equation*}
\lim_{N\to\infty} \sum_{\xi\in E_N} R^{(8)}_N(t,\delta,\xi) \;\le\; 
\lim_{N\to\infty} \max_{y\in S} \sup_{\eta\in\ms E_N^y}
\bb P_{\eta} \big[ H_{\ms B_N^y} > \delta/2 \big] \;=\; 0\;.
\end{equation*}

In view of the definition of $p(\eta)$, the first term in \eqref{b12}
can be written as
\begin{equation*}
\sum_{y\in S} \bb P_{\eta^N}[ X_N(t) =y] \, 
\mu^y_N(\xi) \;+\; R^{(9)}_N(t,\delta,\xi)\;, 
\end{equation*}
where
\begin{equation*}
R^{(9)}_N(t,\delta,\xi) \;=\; \sum_{y\in S} 
\Big\{ \bb P_{\eta^N}[ X_N(t-\delta) =y] - \bb P_{\eta^N}[ X_N(t) =y] \Big\}\, 
\mu^y_N(\xi) \;.
\end{equation*}
Clearly, $\sum_{\xi \in E_N} |R^{(9)}_N(t,\delta,\xi)|$ is less than
or equal to
\begin{align*}
& \sum_{y\in S} 
\Big\{ \bb P_{\eta^N} \big[ X_N(t-\delta) =y \,,\, X_N(t) \not =y \big] \,+\,
\bb P_{\eta^N} \big[ X_N(t-\delta) \not = y \,,\,  X_N(t) =y \big] \Big\} \\
&\quad \le \; 2 \sum_{y\in S} 
\bb P_{\eta^N} \Big[ \sup_{|s-r|\le \delta} |X^T_N(r) - X^T_N(s)| \ge
1 \Big] \;+\; \sum_{u=t,t-\delta} \bb P_{\eta^N}[ X_N(u) =0] \;,
\end{align*}
where the supremum is carried over real numbers $r$, $s$ in $[0,t]$. 
By assumption (H1) and Corollary \ref{l07}, 
\begin{equation*}
\lim_{\delta\to 0} \limsup_{N\to\infty}  \sum_{\xi \in E_N}
|R^{(9)}_N(t,\delta,\xi)|  \;=\; 0\;. 
\end{equation*}

Up to this point we proved that
\begin{equation}
\label{b13}
[\delta_{\eta^N} S^N(t)](\xi) \;=\;
\sum_{y\in S} \bb P_{\eta^N}[ X_N(t) =y] \, 
\mu^y_N(\xi) \;+\; R_N(t,\delta,\xi)\;, 
\end{equation} 
where 
\begin{equation}
\label{b15}
\lim_{\delta\to 0} \limsup_{N\to\infty} 
\sum_{\xi\in E_N} \big| R_N(t,\delta,\xi) \big| \;=\; 0 \;.
\end{equation}
Therefore, in view of \eqref{b13},
\begin{align*}
& \big\Vert \delta_{\eta^N} S^N(t) - \sum_{y\in S}
\bs P_{x}[ \bs X(t) =y] \, \mu^y_N \big\Vert_{\rm TV} \\
&\quad \;=\;
\frac{1}{2}\sum_{\xi\in E_N} \big\vert \delta_{\eta^N} S^N(t) (\xi) - \sum_{y\in S}
\bs P_x [ \bs X(t) =y] \, \mu^y_N(\xi) 
\big\vert \\
&\quad \;\le\; \frac{1}{2}\sum_{y\in S} \big|\, 
\bs P_x [ \bs X(t) =y]  - \bb P_{\eta^N} [ X_N(t) =y] \, \big| 
\;+\; \frac{1}{2}\sum_{\xi\in E_N} \big| R_N(t,\delta,\xi) \big|\;,
\end{align*}
which completes the proof of the proposition, in view of \eqref{b15}
and Corollary \ref{l07}.
\end{proof}

\begin{lemma}
\label{3cond}
Under ({H1}), ({M1}), ({M2}), \eqref{b09a} and any of the assumptions
{\em (a), (b)} or {\em (c)} of Proposition \ref{p02}, for any $y\in S,
\, s\in(\delta/2,\delta)$ we have
\begin{equation*}
\lim_{\delta\to 0}\limsup_{N\to\infty}\sum_{\xi\in\ms E_N^y} \Big|\,
\bb P_{\mu_N^y}^N\big[ \xi^N(s)=\xi\big]-\mu_N^y(\xi)\Big|\;=\;0\;.
\end{equation*}
\end{lemma}

\begin{proof}
For all $\xi\in\ms E_N^y$,
\begin{align*}
\bb P_{\mu_N^y}^N\big[ \xi^N(s)=\xi\big] 
\;&=\; \frac{1}{\mu_N(\ms E_N^y)} \sum_{\zeta\in\ms E_N^y} \mu_N(\zeta)\,
\bb P_\zeta^N\big[ \xi^N(s)=\xi\big] \\
&=\; \frac{1}{\mu_N(\ms E_N^y)} \, \bb P_{\mu_N}^N\big[
\xi^N(s)=\xi\big] \; -\;
\sum_{\zeta\notin \ms E_N^y }\frac{\mu_N(\zeta)}{\mu_N(\ms E_N^y)} 
\, \bb P_\zeta^N\big[ \xi^N(s)=\xi\big]\\
&=\; \mu_N^y(\xi) \;-\; \sum_{\zeta\notin \ms E_N^y}
\frac{\mu_N(\zeta)}{\mu_N(\ms E_N^y)} \,\bb P_\zeta^N\big[ \xi^N(s)=\xi\big]\, .
\end{align*}
Hence,
\begin{equation}
\label{b101}
\begin{aligned}
& \sum_{\xi\in\ms E_N^y} \Big|\,\bb P_{\mu_N^y}^N\big[
\xi^N(s)=\xi\big] \, -\, \mu_N^y(\xi) \Big| \;=\;
\sum_{\zeta\notin \ms E_N^y}\frac{\mu_N(\zeta)}{\mu_N(\ms E_N^y)} 
\, \bb P_\zeta^N\big[ \xi^N(s)\in\ms E_N^y\big] \\
& \qquad \le \; \frac{\mu_N(\Delta_N)}{\mu_N(\ms E_N^y)}
\;+\; \frac{1}{\mu_N(\ms E_N^y)} \sum_{\zeta\in \breve{\ms E}_N^y}
\mu_N(\zeta)\, \bb P_\zeta^N\big[ \xi^N(s)\in \ms E_N^y\big]\; .
\end{aligned}
\end{equation}
By \eqref{b09a}, the first term of this sum vanishes, as
$N\to\infty$. It remains to show that the second term also vanishes
under assumption (a), (b) or (c).

Assume first that (a) holds. Then, by reversibility, the last term in
\eqref{b101} is equal to
\begin{equation*}
\sum_{\xi\in\ms E_N^y} \mu_N^y(\xi)\, \bb P_{\xi}^N
\big[\xi^N(s)\in\breve{\ms E}_N^y\big] \;\le\; 
\bb P_{\mu_N^y} \big[ \sup_{0\le s\le\delta} |X^T_N(s) - y| \ge 1 \big]\;.
\end{equation*}
This expression vanishes, as $N\to\infty$, by assumption (H1). This
completes the proof of the lemma under the hypothesis (a).

Assume now that condition (b) is in force. In this case, the last term in
\eqref{b101} is bounded by
\begin{equation*}
\sum_{z\neq y}\frac{\mu_N(\ms E_N^z)}{\mu_N(\ms E_N^y)}\,
\bb P_{\mu_N^z} \Big[ \sup_{0\le s\le\delta} |X^T_N(s) - z| \ge 1 \Big] 
\;\le\; C_0 \sum_{z\neq y}\bb P_{\mu_N^z} \Big[ \sup_{0\le s\le\delta}
|X^T_N(s) - z| \ge 1 \Big]\;. 
\end{equation*}
Here again, by assumption (H1), this expression vanishes, as
$N\to\infty$. This completes the proof of the lemma under the
hypothesis (b).

Assume, finally, that condition (c) is fulfilled. Note that
\begin{align*}
\bb P_{\mu_N^y}^N\big[ T_N({s-\varepsilon_N})-T_N({s-2\varepsilon_N})&\le \frac{1}{2}\varepsilon_N\big] \le \bb P_{\mu_N^y}^N\big[\int_{s-2\varepsilon_N}^{s-\varepsilon_N}\mathbf{1}\{\xi^N(t)\in\Delta_N\}\, dt \ge \frac{1}{2}\varepsilon_N\big]\\
&\le 2\,\frac{\mu_N(\Delta_N)}{\mu_N(\ms E_N^y)},
\end{align*}
by Markov's inequality. The last expression vanishes as $N\to\infty$ by \eqref{b09a}.
Define the stopping
time $\sigma_N$ as
\begin{equation*}
\sigma_N \;=\; \inf \Big\{t\ge s-2\varepsilon_N: \xi^N(t)\in \ms
B_N^y\Big\}\;. 
\end{equation*}
By repeating the arguments that led to \eqref{tracewilldo} and
\eqref{iterated} we obtain that
\begin{equation}
\label{b16}
\lim_{\delta\to 0}\lim_{N\to\infty}\bb
P_{\mu_N^y}^N\big[\sigma_N>s-\varepsilon_N\big] \;=\;0\;.
\end{equation}
Let
\begin{equation*}
R_N^{(10)}(s,\delta,\xi) \;=\;
\bb P_{\mu_N^y}^N\big[ \xi^N(s)=\xi\big] \;-\; \mu_N^y(\xi)\;.
\end{equation*}
Conditioning first on $\sigma_N$, and using \eqref{b04},
\eqref{b07} and \eqref{b16} yields that
\begin{equation*}
\lim_{\delta\to 0}\limsup_{N\to\infty}\sum_{\xi\in\ms
  E_N^y}|R_N^{(10)}(s,\delta,\xi)| \;=\;0\;.
\end{equation*}
This concludes the argument.
\end{proof}

\section {Examples}
\label{sec02}

We present in this section four examples to evaluate the conditions
introduced in the previous sections. The first example belongs to the
class of models in which the metastable sets are singletons. In the
second and third examples the metastable sets are not singletons, but
the process visits all configurations of a metastable set before
hitting a new metastable set. These processes are said to visit
points. In the second example the assumptions of Lemma \ref{l08} are
in force, but not in the third. For this latter class, we show that
the conditions of Corollary \ref{l07} are fulfilled for an appropriate
singleton set $\ms B^x_N$. In the last example, the process does not
visit all configurations of a metastable set before reaching a new
metastable set. In these models the entropy plays an important role in
the metastable behavior of the system. For this last model, we prove
that the hypotheses of Lemma \ref{l04} hold.

The purpose of this section is not to show that the conditions of
Lemmata \ref{l08}, \ref{l04} or Corollary \ref{l07} are in force in
great generality. Actually, in some cases, this requires lengthy
arguments and a detailed analysis of the dynamics.  We just want to
convince the reader that this is possible. In other words, that one
can deduce the convergence of the finite-dimensional distributions and
the convergence of the state of the process from conditions (H1), (H2) and some reasonable additional
conditions.

In the arguments below we use the Dirichlet and the Thomson principles
for the capacities between two disjoint sets of $E_N$. We do not
recall these results here and we refer to \cite[Section 7.3]{BovHol} 

\begin{example}
[Inclusion process \cite{grv, bdg}]
\label{ex1b}
The inclusion process describes the evolution of particles on a
countable set.  Recall from \eqref{f02} that we denote by $\bb T_L$,
$L\ge 1$, the discrete, one-dimensional torus with $L$ points, by
$E_N$ the set of configurations on $\bb T_L$ with $N$ particles, and
by $\eta_x$, $x\in \bb T_L$, the total number of particles at $x$ for
the configuration $\eta$.

Fix a sequence $(d_N : N\ge 1)$ of strictly positive numbers. Recall
from \eqref{f01} the definition of the configuration
$\sigma^{x,y}\eta$. The reversible, nearest-neighbor, inclusion
process associated to the sequence $d_N$ is the continuous-time,
$E_N$-valued Markov process $\{\eta^N(t) : t\ge 0\}$ whose generator
$L_N$ acts on functions $f: E_N\to\bb R$ as
\begin{equation*}
(L_N f) (\eta) \;=\; \sum_{\stackrel{x,y\in \bb T_L}{x\not = y}}
\eta_x\, ( d_N + \eta_y) \, r(y-x) \, 
\big\{ f(\sigma^{x,y}\eta) - f(\eta) \big\} \;,
\end{equation*}
where $r(-1) = r(1) = 1$, $r(x)=0$, otherwise.

The inclusion process is clearly irreducible and it is reversible with
respect to the probability measure $\mu_N$ given by
\begin{equation*}
\mu_N(\eta) \;=\; \frac {1} {Z_{N}} \, 
\prod_{x\in \bb T_L} w_N(\eta_x) \;,
\end{equation*}
where $Z_{N}$ is the normalizing constant, $w_N(k) = \Gamma
(k+d_N)/k!\, \Gamma(d_N)$, and $\Gamma$ is the gamma function.

Assume that $d_N \log N \to 0$, as $N\uparrow\infty$. Denote by
$\xi^{x,N}$ the configurations in which all particles are placed at
site $x$, $\xi^{x,N}_x = N$, $\xi^{x,N}_y = 0$ for $y\not = x$, and
let $\ms E^x_N = \{\xi^{x,N}\}$. By \cite[Proposition 2.1]{bdg},
$\mu_N(\ms E^x_N) \to 1/L$ as $N\uparrow\infty$.

The metastable behavior of the inclusion process in the sense of
conditions (H1), (H2) has been proved in \cite[Theorem 2.3]{bdg}. The
time-scale at which a metastable behavior is observed is given by
$\theta_N = 1/d_N$.
\end{example}

In this model the metastable sets $\ms E^x_N$ are singletons.  This
phenomenon occurs in many other models. For instance, in spin systems
evolving in large, but fixed, volumes as the temperature vanishes
(cf. the Ising model with an external field under the Glauber dynamics
\cite{nevsch, sch1, bl2011} and the Blume-Capel model with zero
chemical potential and a small magnetic field \cite{CNS2015, ll,
  CNS2016}).  It also occurs for random walks evolving among random
traps \cite{jlt1, jlt2}.

We claim that all hypotheses of Propositions \ref{l01}, \ref{p02} are
in force. Actually, with the exception of (H1) and (H2), all
assumptions trivially hold because the metastable sets are singletons.
Set $\ms B^x_N = \ms E^x_N = \{\xi^{x,N}\}$.

\smallskip\noindent {\sl A. Conditions (H1) and (H2)}. We already
mentioned that assumptions (H1) and (H2) have been proved in
\cite{bdg} with the time-scale $\theta_N = 1/d_N$.

\smallskip\noindent {\sl B. Condition \eqref{03}}. By
\cite[Proposition 2.1]{bdg}, $\mu_N(\xi^{x,N})\to 1/L$. In particular,
the inclusion process satisfies the assumption of Lemma \ref{l08}.

\smallskip\noindent {\sl C. Condition (M1)}.  Condition (M1) is empty
because the sets $\ms E^x_N$ and $\ms B^x_N$ coincide.

\smallskip\noindent {\sl D. Condition \eqref{b04} of (M2)}. Since $\ms
E^x_N = \{\xi^{x,N}\}$, starting from $\xi^{x,N}$, $H_{\Delta_N}$
corresponds to the first jump of the Markov chain $\xi^N(t)$, denoted
hereafter by $\tau_1$: $\bb P_{\xi^x_N} [H_{\Delta_N} =
\tau_1]=1$. Since the process has been speeded-up by $\theta_N =
1/d_N$, $\tau_1$ is an exponential random variable of rate $2N$. It is
thus enough to choose a sequence $\varepsilon_N$ such that
$\varepsilon_N \ll 1/N$.

\smallskip\noindent {\sl E. Condition \eqref{b07} of (M2)}. This
condition is empty because $\ms E^x_N = \{\xi^{x,N}\}$. It holds for
any sequence $\varepsilon_N>0$.

\smallskip\noindent {\sl F. Condition \eqref{b09a} of Proposition
  \ref{p02}}. This is a consequence of \cite[Proposition 2.1]{bdg}
which asserts that $\mu_N(\xi^{x,N})\to 1/L$.

\smallskip\noindent {\sl G. Conditions (a), (b) or (c)}.  Assumption
(a) of Proposition \ref{p02} is in force as the process is reversible.

\begin{example}[Condensing zero-range processes \cite{bl2012a, l2014, s2018}]
\label{ex3}
This model has been introduced at the beginning of Section
\ref{sec-2}.  Set $\theta_N = N^{1+\alpha}$.
\end{example}

The condensing zero-range process is an example of a process which
visits points in the sense that, starting from a well $\ms E^x_N$, the
dynamics visits all configurations of $\ms E^x_N$ before reaching
another well.  This property reads as follows.  For all $x\in S = \bb
T_L$,
\begin{equation*}
\lim_{N\to\infty} \max_{\eta, \xi\in \ms E^x_N}
\bb P^N_\eta [ H_{\breve{\ms E}^x_N} < H_\xi ]  \;=\; 0\;,
\end{equation*}
where $\breve{\ms E}^x_N$ has been introduced in \eqref{17}.
Other examples of metastable dynamics which visit points are random
walks in a potential field \cite{cgov, begk01, lmt2015, lseo2016a}.

We show below that all hypotheses of Propositions \ref{l01}, \ref{p02}
are in force. In certain cases we impose further assumptions on the
dynamics, e.g., that it is reversible or that $|S|=2$, to avoid
lengthy arguments. The main tool to prove this assertion is the fact
that the process visit points. Recall from \eqref{l3-19} that we
denote by $\Cap_{N} (\ms A, \ms B)$ the capacity between two disjoint
subsets $\ms A$ and $\ms B$ of $E_N$. Since $\xi^N(t)$ is the process
$\eta^N(t)$ speeded-up by $\theta_N$, by \cite[Theorem 2.2]{bl2012a}
and \cite[Theorem 6.3]{s2018}, for any disjoint subsets $A$, $B$ of
$S$,
\begin{equation}
\label{10}
\lim_{N\to\infty} \Cap_{N} \Big( \bigsqcup_{x\in A} \ms E^x_N,
\bigsqcup_{y\in B} \ms E^y_N \Big) \;=\; C(A,B) \;\in\; (0,\infty)\;,
\end{equation}
where $C(A,B)$ is the capacity between $A$ and $B$ for the random walk
on $S$ with transition probabilities $p(y\!-\!x)$, for $x,y\in S$.\\[2mm]
\smallskip\noindent{\sl A. Conditions (H1) and (H2)}. Assumptions (H1)
and (H2) have been proved in \cite{bl2012a} in the reversible case, in
\cite{l2014} in the totally asymmetric case, $p=1$, and in
\cite{s2018} in the asymmetric case $1/2<p<1$. 

\smallskip\noindent{\sl B. Condition \eqref{03}}. We prove that the
assumptions of Lemma \ref{l04} are in force in the reversible case for
$\ms B^x_N = \{\xi^{x,N}\}$, where $\xi^{x,N}$ represents the
configurations in which all particles are placed at site $x$.

Fix $x\in S$ and $\eta\in\ms E^x_N$. By the Markov inequality and
\cite[Proposition 6.10]{bl2},
\begin{equation*}
\bb P_\eta \big[H_{\ms B_N^x}>\delta\big] \;\le\;
\frac 1 \delta\, \bb E_\eta \big[H_{\ms B_N^x}\big]
\;\le\;
\frac 1 \delta\, \frac 1{\Cap_N(\eta , \xi^{x,N})} 
\end{equation*}
By (H1), page 806 in \cite{bl2012a},
\begin{equation*}
\lim_{N\to\infty} \sup_{\eta\in\ms E_N^x} 
\frac {\Cap_N(\ms E^x_N, \breve{\ms E}^x_N )}{\Cap_N(\eta ,
  \xi^{x,N})} \;=\; 0\;. 
\end{equation*}
Therefore, by \eqref{10}, for every $\delta>0$, 
\begin{equation}
\label{l3-01}
\lim_{N\to\infty} \max_{\eta\in\ms E^x_N} \bb P_\eta
\big[H_{\ms B_N^x}>\delta\big] \;=\; 0\;.
\end{equation}

On the other hand, for every $s>0$,
\begin{equation*}
\bb P_{\xi^{x,N}} \big[ \xi^N(s) \in \Delta_N \big] \;\le\;
\frac 1{\mu_N(\xi^{x,N})} \bb P_{\mu_N} 
\big[ \xi^N(s) \in \Delta_N \big] \;=\;
\frac {\mu_N(\Delta_N)}{\mu_N(\xi^{x,N})} \;\cdot
\end{equation*}
By equation (3.2) in \cite{bl2012a}, $\mu_N(\Delta_N) \to 0$, and by
\cite[Proposition 2.1]{bl2012a}, $\mu_N(\xi^{x,N}) \to 1/Z_S>0$. This
shows that the second assumption of Lemma \ref{l04} is in force.

I. Seo extended the previous result to the asymmetric case $1/2<p<1$
in \cite[Proposition 6.3]{s2018}.

\smallskip\noindent{\sl C. Condition (M1)}. Since $H_{\ms B_N^x}^{\ms E_N^x} \le
  H_{\ms B_N^x}$, condition (M1) follows from \eqref{l3-01}.

\smallskip\noindent{\sl D. Condition \eqref{b04} of (M2)}. Since the
exterior boundary of $\ms E^x_N$ is contained in $\Delta_N$, under
$\bb P_{\xi^{x,N}}$,  $H_{(\ms E^x_N)^c}=H_{\Delta_N}$. We claim that
\begin{equation}
\label{l3-06}
\bb P_{\xi^{x,N}} [H_{\Delta_N} \le 2\, \varepsilon_N] \;\le\;
\frac{C_0\, \varepsilon_N\, \theta_N}{\ell^\alpha_N}    
\end{equation}
for some finite constant $C_0$. In particular, condition \eqref{b04}
of (M2) is fulfilled provided we choose $\varepsilon_N\, \theta_N \ll
\ell^\alpha_N$. 

We turn to the proof of \eqref{l3-06}. By Corollary \ref{lv3-l3} we have 
\begin{equation}
\bb P_{\xi^{x,N}} [H_{\Delta_N} \le 2\, \varepsilon_N] 
\;\le\; \frac{2e\, \varepsilon_N}{\mu_N(\xi^{x,N})}
\Cap_N(\xi^{x,N} , \Delta_N).
\label{v4-bc}
\end{equation}
On the other hand, by monotonicity of capacities
\begin{equation*}
\Cap_N (\xi^{x,N} , \Delta_N) \;\le\; \Cap_N(\ms E^x_N, \Delta_N) \;=\; \frac {1}2
\sum_{\eta\in \ms E^x_N} \sum_{\xi \in \Delta_N} \mu_N(\eta)\, R_N(\eta,\xi) \;.
\end{equation*}

Since the holding rates $\lambda_N(\eta)$ are uniformly bounded by $C_0 \theta_N$, 
if we denote by $\partial \ms E^x_N$ the interior boundary of the set $\ms
E^x_N$, the previous sum is bounded by $C_0\, \theta_N\,
\mu_N(\partial \ms E^x_N)$. An explicit computation shows that the
measure of $\partial \ms E^x_N$ is bounded by
$\ell_N^{-\alpha}$. The proof of this assertion is similar to the one
of \cite[Lemma 3.1]{bl2012a} and is omitted. Hence, $\Cap_N (\xi^{x,N}
, \Delta_N) \;\le\; C_0\, \theta_N\, \ell_N^{-\alpha}$. Together with \eqref{v4-bc} and 
\cite[Proposition 2.1]{bl2012a}, this gives \eqref{l3-06}. (Remark: In
the case $|S|=2$, it is possible to compute exactly $\Cap_N (\xi^{x,N}
, \Delta_N)$ and one gets that it is of order $\theta_N\,
\ell_N^{-(1+\alpha)}$. We lost a factor $1/\ell_N$ at
the first estimate in the preceding display.)

\smallskip\noindent{\sl E. Condition \eqref{b07} of (M2)}. The proof
relies on an estimate of the spectral gap. We prove this condition in
the case of two sites, the general case can be handled using the
martingale approach developed by Lu and Yau \cite[Appendix 2]{kl}.

Assume that $|S|=2$, and denote by $\lambda_{R,1}$ the spectral gap of
the process $\xi^N(t)$ reflected at $\ms E^1_N = \{0, \dots,
\ell_N\}$. We claim that
\begin{equation}
\label{l3-10}
\lambda_{R,1} \;\ge\; \frac{c_0\, \theta_N}{\ell^2_N}\;\cdot 
\end{equation}

On two sites, the zero-range process is a birth and death process, and
the reflected process on $\ms E^1_N$ is the continuous-time Markov
chain whose generator is given by
\begin{equation*}
(\mc L^{R,1}_N f)(\eta) \;=\; g_{R,N} (\eta) \{f(\eta - 1) - f(\eta)\} \;+\;
g_{R,N}(N-\eta) \{f(\eta + 1) - f(\eta)\} \;, \quad \eta\;\in\;
\ms E^1_N\;,
\end{equation*}
where $g_{R,N}(\zeta) = \theta_N\, g(\zeta)$ for all $\zeta\not =
N-\ell_N$, and $g_{R,N}(N - \ell_N)=0$, due to the reflection at
$\ms E^1_N$.  Denote by $\mu^1_N$ the stationary measure $\mu_N$
conditioned to $\ms E^1_N$.

In order to prove \eqref{l3-10}, we have to show that there exists a
finite constant $C_0$ such that
\begin{equation}
\label{l3-08}
E_{\mu^1_N}\Big[ \big(f - E_{\mu^1_N} [f] \big)^2\Big] 
\;\le\; C_0 \, \frac{\ell^2_N}{\theta_N}\, \< f, (-\mc L^{R,1}_N) f\>_{\mu^1_N}
\end{equation}
for all $N\ge 1$ and all functions $f: \{0, \dots, \ell_N\} \to \bb
R$, where $\< f, g\>_{\mu^1_N}$ represents the scalar product in
$L^2(\mu^1_N)$.

Fix a function $f: \{0, \dots, \ell_N\} \to \bb R$. By Schwarz
inequality, 
\begin{align*}
& E_{\mu^1_N}\Big[ \big(f - E_{\mu^1_N} [f] \big)^2\Big] 
\;\le\; E_{\mu^1_N}\Big[ \big(f - f(0) \big)^2\Big] \\
&\quad \;\le\; \sum_{\eta\in \ms E^1_N} \mu^1_N(\eta)
\sum_{\xi =0}^{\eta-1} [f(\xi+1) - f(\xi)]^2 \, \mu^1_N(\xi)\, 
\sum_{\xi' =0}^{\eta-1} \frac 1{\mu^1_N(\xi')}\;\cdot
\end{align*}
The sum over $\xi'$ is bounded by $C_0 \eta^{1+\alpha}$. Hence, since
$\mu^1_N(\eta) \le C_0 \eta^{-\alpha}$, changing the order of
summations the previous expression is seen to be less than or equal to
\begin{equation*}
C_0 \sum_{\xi =0}^{\ell_N-1}
[f(\xi+1) - f(\xi)]^2 \, \mu^1_N(\xi)\, 
\sum_{\eta=\xi+1}^{\ell_N} \eta \;\le\;
C_0 \, \ell^2_N \sum_{\xi =0}^{\ell_N-1}
[f(\xi+1) - f(\xi)]^2 \, \mu^1_N(\xi) \;.
\end{equation*}
This expression is bounded by $C_0 \, (\ell^2_N/\theta_N)\, \< f,
(-\mc L^{R,1}_N) f\>_{\mu^1_N}$ because $g$ is bounded below by a
positive constant and the process is speeded-up by $\theta_N$. This
proves claim \eqref{l3-08}, and therefore \eqref{l3-10}.

We turn to condition \eqref{b07} of (M2). We claim that this condition
is fulfilled provided $\varepsilon_N \theta_N \gg \ell^{2}_N$. Indeed,
since $\mu^y_N({\xi^{x,N}}) \ge c_0$, in view of \eqref{l3-07}, we have to
show that
\begin{equation}
\label{l3-09}
\lim_{N\to \infty} \lambda_{R,1} \, \varepsilon_N\;=\; \infty \;,
\end{equation}
which follows from \eqref{l3-10} if $\varepsilon_N \theta_N \gg
\ell^{2}_N$.

For $|S|=2$, in view of (D) and (E) above, conditions \eqref{b04} and
\eqref{b07} of (M2) are fulfilled for any sequence $\varepsilon_N$
such that $\ell^{2}_N \ll \varepsilon_N \theta_N \ll \ell^{1+\alpha}_N$.

\smallskip\noindent{\sl F. Condition \eqref{b09a} of Proposition
  \ref{p02}}. By \cite[Remark 2.5]{bl2012a},
\begin{equation*}
\lim_{N\to\infty} \frac {\mu_N(\Delta_N)}{\mu_N(\ms E^x_N)} \;=\; 0\;.
\end{equation*}

\smallskip\noindent{\sl G. Conditions (a), (b) or (c)}.  Assumption (b) of
Proposition \ref{p02} is in force since $\mu_N(\ms E^x_N)= \mu_N(\ms
E^y_N)$ for all $x$, $y\in S$.

\begin{example}[Random walk in a potential field] 
\label{ex2}
In this example, the sets $\ms B^x_N$ are still reduced to singletons,
$\ms B^x_N=\{\xi^{x,N}\}$, but $\mu_N(\xi^{x,N})\to 0$. To simplify
the discussion as much as possible, we assume that the process is
reversible and that the potential has two wells of the same height,
but the arguments apply to the more general situations considered in
\cite{begk01, lmt2015, lseo2016a}.

Let $\Xi$ be an open, bounded and connected subset of $\bb R^d$ with a
smooth boundary $\partial \,\Xi$. Fix a smooth function
$F:\Xi\cup \partial \,\Xi\to\bb R$, with three critical points,
satisfying the following assumptions:

\begin{enumerate}
\item[(RW1)] There are two local minima, denoted by $m_1$, $m_2$. All
  the eigenvalues of the Hessian of $F$ at these points are
  \emph{strictly} positive. Moreover, $F( m_1)=F( m_2)=:h$.
\item[(RW2)] The other critical point of $F$ is denoted by $
  \sigma$. The Hessian of $F$ at $\sigma$ has one strictly
  negative eigenvalue, all the other ones being strictly positive.
\item[(RW3)] For every $x\in \partial\, \Xi$, $(\nabla F)(x)
  \cdot n (x) > 0$, where $n (x)$ represents the
  exterior normal to the boundary of $\Xi$, and $x\cdot y$ the
  scalar product of $x$, $y\in\bb R^d$. This hypothesis
  guarantees that $F$ has no local minima at the boundary of $\Xi$.
\end{enumerate}

Denote by $\Xi_N$ the discretization of $\Xi$: $\Xi_N = \Xi \cap
(N^{-1} \bb Z)^d$, $N\ge 1$.  Let $\mu_N$ be the probability measure
on $\Xi_N$ defined by
\begin{equation*}
\mu_N(\eta) \;=\; \frac 1{Z_N} e^{-N F(\eta)}\;,\quad \eta \in
\Xi_N\;, 
\end{equation*}
where $Z_N$ is the partition function $Z_N = \sum_{\eta\in \Xi_N}
\exp\{-N F(\eta)\}$. By equation (2.3) in \cite{lmt2015},
\begin{equation}
\label{ex4-1}
\lim_{N\to\infty} \frac{Z_N e^{N h}}{(2\pi N)^{d/2}  }
\;=\; \frac 1{\sqrt{\det {\rm Hess}\, F(m_1)}} \;+\;
\frac 1{\sqrt{\det {\rm Hess}\, F(m_2)}}\;, 
\end{equation}
where ${\rm Hess}\, F(x)$ represents the Hessian of $F$ calculated
at $x$ and $\det {\rm Hess}\, F(x)$ its determinant.

Let $\{\eta^N(t) : t\ge 0\}$ be the continuous-time
Markov chain on $\Xi_N$ whose generator $L_N$ is given by
\begin{equation}
\label{v54}
(L_N f)(\eta) \;=\; \sum_{\substack{\xi \in\Xi_N\\
\Vert\xi - \eta \Vert = 1/N}}  e^{-(1/2) N [F(\xi) - F(\eta)]} 
\, [f(\xi) - f(\eta)]\;,
\end{equation}
where $\Vert \,\cdot\,\Vert$ represents the Euclidean norm of $\bb
R^d$.  
\end{example}

Recall that $m_i$, $i=1$, $2$, represent the two local minima of $F$
in $\Xi$, and $\sigma$ the saddle point. Let
$H:=F(\sigma)>F(m_1)=F(m_2)=h$.  Denote by $V_i = B_\kappa(m_i)$,
$\kappa>0$, two balls of radius $\kappa$ centered at the local minima.
Assume that $\kappa$ is small enough for $\sup_{x\in V_i} F(x) < H$.
Denote by $\ms E^i_N$ the discretization of the sets $V_i$: $\ms E^i_N
= \Xi_N \cap V_i$.

Let $\theta_N = 2\pi N \exp\{[H-h]N\}$. It has been proved in
\cite{lmt2015, lseo2016a} that the process $X^{\rm T}_N(t)$ fulfills
conditions (H1) and (H2).  We claim that the assumptions of
Propositions \ref{l01} and \ref{p02} are in force.

We prove condition \eqref{03} through Corollary \ref{l07} with $\ms
B^i_N = \{\xi^{i,N}\}$, where $\xi^{i,N}$ is a point in $\Xi_N$ which
approximates the local minima $m_i$.

\smallskip\noindent{\sl A. Condition \eqref{b02}}.
Fix $\eta\in \ms E^i_N$. Since $H_{\ms B_N^i}^{\ms E_N^i}\le H_{\ms B_N^i}$, by the Markov inequality, it is enough to
prove that
\begin{equation}
\label{l3-13}
\lim_{N\to\infty} \bb E_\eta [H_{\ms B_N^i}] \;=\; 0\;.
\end{equation}
By \cite[Proposition 6.10]{bl2}, the expectation is bounded by
$1/\Cap_N (\eta , \ms B_N^i)$. Consider a path $(\eta_0 = \eta,
\eta_1, \dots, \eta_M = \xi^{i,N})$ such that $M\le C_0 N$,
$\eta_i\in\Xi_N$, $\Vert \eta_i - \eta_{i+1}\Vert = 1/N$, $F(\eta_i)
\le H - \epsilon$ for some $\epsilon>0$.  Let $\Phi$ be the unitary
flow from $\eta$ to $\xi^{i,N}$ such that $\Phi(\eta_i, \eta_{i+1})
=1$. By Thomson's principle,
\begin{equation*}
\frac 1{\Cap_N (\eta , \xi^{i,N})} \;\le\; \frac{Z_N}{\theta_N}
\sum_{j=0}^{M-1}  e^{(N/2) \, [F(\eta_i) + F(\eta_{i+1})]}\;.
\end{equation*}
The factor $\theta_N$ appeared as the process has been
speeded-up. This expression vanishes as $N\to\infty$ in view of
\eqref{ex4-1}, the definition of $\theta_N$, and because $F(\eta_i)\le
H-\epsilon$, $M\le C_0N$.

\smallskip\noindent{\sl B. Condition \eqref{b04}.} Let $h_i =
\inf_{x\in\partial V_i} F(x)$. We claim that this condition is in
force provided
\begin{equation*}
\varepsilon_N\, \theta_N
\;\ll\; N^{-d} \, e^{ N [h_i - h]} \;.
\end{equation*}

Since, under $\bb P_{\xi^{i,N}}$, $H_{(\ms E^i_N)^c} = H_{\Delta_N}$,
we need to estimate $\bb P_{\xi^{i,N}} [ H_{\Delta_N} \le
2\varepsilon_N]$.  By Corollary \ref{lv3-l3},
\begin{equation*}
\bb P_{\xi^{i,N}} \big[ H_{\Delta_N} \le 2\varepsilon_N \big] 
\;\le\; \frac{C_0 \, \varepsilon_N\, \theta_N}{\mu_N(\xi^{i,N})}
\frac 1{Z_N}\, \sum_{\substack{\eta \in \partial_- \ms E^i_N \,,\,
\zeta \in \Delta_N \\
\Vert \eta-\zeta\Vert = 1/N} }
 e^{- (N/2) \, [F(\eta) + F(\zeta)]}\;,
\end{equation*}
where $\partial_- \ms E^i_N$ stands for the inner boundary of $\ms
E^i_N$: 
\begin{equation*}
\partial_- \ms E^i_N \;=\; \big \{\eta\in \ms E^i_N : 
\mu(\eta) \, R(\eta, \xi) >0
\text{ for some $\xi\not\in\ms E^i_N$ } \big\}\;.
\end{equation*}
By definition of $\ms E^i_N$, the right-hand side of the penultimate
formula is bounded above by $C_0 \, \varepsilon_N\, \theta_N\, N^d\, 
\exp\{ - N [h_i - h]\}$, which proves the claim.

\smallskip\noindent{\sl C. Condition \eqref{b07}.} We claim that this
condition is fulfilled provided 
\begin{equation}
\label{l3-12}
\theta_N \, \varepsilon_N  \;\gg\;  N^{d+1+b}
\end{equation}
for some $b>0$. 

We first estimate the spectral gap of the reflected process
$\xi^N_{R,i}(t)$, denoted by $\lambda_{R,i}$. We claim that
$\lambda_{R,i} \ge c_0 \, \theta_N \, N^{-(d+1)}$. To prove this
assertion, we have to show that
\begin{equation}
\label{l3-11}
E_{\mu^i_N}\Big[ \big(f - E_{\mu^i_N} [f] \big)^2\Big] 
\;\le\; C_0 \, \frac{N^{d+1}}{\theta_N}\, 
\< f, (-\mc L^{R,i}_N) f\>_{\mu^i_N}
\end{equation}
for all $N\ge 1$ and all functions $f: \ms E^i_N \to \bb R$, where $\<
f, g\>_{\mu^i_N}$ represents the scalar product in $L^2(\mu^i_N)$. For
each $\eta \in \ms E^i_N$, denote by $\gamma(\eta) = (\eta_0=\eta,
\dots, \eta_M =\xi^{i,N})$ a discrete version of the path from $\eta$
to $\xi^{i,N}$ given by $\dot x(t) = - (\nabla F)(x (t))$. This means
that $\Vert \eta_{j+1} - \eta_j\Vert =\frac{1}{N}$, $M\le C_0 N$, and
$\eta_j$ is the closest point of the lattice $\Xi_N$ to $x(t_j)$ for
some increasing sequence of times $\{t_j\}_{0\le j\le M}$.  Clearly,
$|F(\eta_j)-F\big(x({t_j})\big)|\le \frac{c_0}{N}$ and since
$\frac{d}{dt}F\big(x(t)\big)=-\|(\nabla F)\big(x(t))\|^2\le 0$, for
all $0\le k \le j\le M$ we have
\[
F(\eta_k)-F(\eta_j)\ge F(\eta_k)-F\big(x(t_k)\big)+F(\big(x(t_j)\big)-F(\eta_j)\ge -\frac{2c_0}{N}.
\]
In particular, 
\begin{equation}
e^{-NF(\eta)}\le e^{2c_0} e^{-\frac{N}{2}\big(F(\eta_j)+F(\eta_{j+1})\big)},\, j=0,1,\ldots,M-1.
\label{Fdec}
\end{equation}
Since $M\le C_0N$, by Schwarz inequality,
\begin{align*}
& E_{\mu^i_N}\Big[ \big(f - E_{\mu^i_N} [f] \big)^2\Big]
\;\le\; E_{\mu^i_N}\Big[ \big(f - f(\xi^{i,N}) \big)^2\Big] \\
&\quad \;\le\; C_0\, N \sum_{\eta\in \ms E^i_N} \mu_N^i(\eta)\,
\sum_{j=0}^{M(\eta)-1} [f(\eta_{j+1}) - f(\eta_j)]^2\\
&\quad \;\le\; C_0\, N\sum_{\eta\in \ms E^i_N}
\sum_{j=0}^{M(\eta)-1} \mu_{N}^i(\eta_j)
R_N(\eta_j,\eta_{j+1})\, [f(\eta_{j+1}) - f(\eta_j)]^2,
\end{align*}
where the last inequality follows from \eqref{Fdec}. Fix an edge
$(\zeta,\zeta')$ and consider all configurations $\eta\in \ms E^i_N$
whose path $\gamma(\eta)$ contains this pair (that is $(\zeta,\zeta')
= (\eta_j, \eta_{j+1})$ for some $0\le j <M$). Of course, there are at
most $|\ms E^i_N| \le C_0 N^d$ such configurations. Hence, changing
the order of summation, the previous sum is seen to be bounded above
by
\begin{equation*}
C_0\, N^{d+1} \sum_{\zeta\in \ms E^i_N} 
\sum_{\substack{\zeta' \in \ms E^i_N \\ 
\Vert \zeta' - \zeta\Vert=1/N}}
\mu_N^i(\zeta)\, R_N(\zeta,\zeta')\, [f(\zeta') - f(\zeta)]^2\;.  
\end{equation*}
This proves claim \eqref{l3-11} since the double sum is equal to
$(2/\theta_N) \< f, (-\mc L^{R,i}_N) f\>_{\mu^i_N}$.

We turn to the proof of condition \eqref{b07}. Fix a sequence
$\varepsilon_N$ satisfying \eqref{l3-12} for some $b>0$. By
\eqref{ex4-1}, $\mu_N(\xi^{i,N}) \ge c_0 N^{-d/2}$. Hence, by
\eqref{l3-11},
\begin{equation*}
\frac 1{\mu^i_N(\xi^{i,N})^{1/2}} \, e^{ - \lambda_{R,i} \varepsilon_N}
\;\le\; C_0\, N^{d/4}\, \exp\big\{ - c_0 \, \theta_N \,
\varepsilon_N  \, N^{-(d+1)}\big\}\;.
\end{equation*}
By \eqref{l3-12} this expression vanishes as $N\to\infty$. This proves
condition \eqref{b07} in view of \eqref{l3-07}.

Conditions \eqref{b09a} and \eqref{b09}
are elementary. Hence, as claimed, all conditions of Propositions
\ref{l01} and \ref{p02} are in force. Similar arguments apply in the
case of several wells and critical points, as well as in the
non-reversible setting.

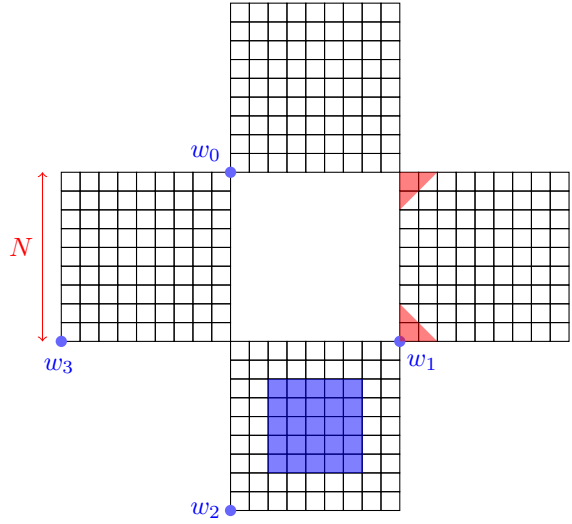
\begin{figure}
\centering
\begin{tikzpicture}[scale = .25]
\foreach \x in {0, ..., 8}
\foreach \y in {0, ..., 8}
\draw (\x,\y) -- (\x,\y+1) -- (\x+1,\y+1) -- (\x+1,\y) -- (\x,\y);
\foreach \x in {9, ..., 17}
\foreach \y in {9, ..., 17}
\draw (\x,\y) -- (\x,\y+1) -- (\x+1,\y+1) -- (\x+1,\y) -- (\x,\y);
\foreach \x in {-9, ..., -1}
\foreach \y in {9, ..., 17}
\draw (\x,\y) -- (\x,\y+1) -- (\x+1,\y+1) -- (\x+1,\y) -- (\x,\y);
\foreach \x in {0, ..., 8}
\foreach \y in {18, ..., 26}
\draw (\x,\y) -- (\x,\y+1) -- (\x+1,\y+1) -- (\x+1,\y) -- (\x,\y);
\draw[<->, red] (-10,9) -- (-10,18);
\draw[red]  (-10, 14) node[anchor=east] {$N$};
\draw[blue] (0,0) node[anchor=east] {$w_2$};
\fill[blue!60!white] (0,0) circle [radius = .3cm];
\draw[blue] (10.2,8.8) node[anchor=north] {$w_1$};
\fill[blue!60!white] (9,9) circle [radius = .3cm];
\draw[blue] (-7.8,7.7) node[anchor=east] {$w_3$};
\fill[blue!60!white] (-9,9) circle [radius = .3cm];
\draw[blue] (0,19) node[anchor=east] {$w_0$};
\fill[blue!60!white] (0,18) circle [radius = .3cm];
\fill [rectangle,blue,opacity=.5] (2,2)--(2,7)--(7,7)--(7,2);
\fill [red,opacity=.5] (9,9)--(11,9)--(9,11)--(9,9);
\fill [red,opacity=.5] (9,18)--(11,18)--(9,16)--(9,18);
\end{tikzpicture}
\caption{The graph $E_N$ of Example \ref{ex4}. The square in blue
  represents the set $\ms B^2_N$ and the red triangles the set
  $\Delta_N \cap Q^1_N$. The figure is misleading because the set $\ms
  B^2_N$ is a square of length $N-2M_N$ and almost fills the set
  $Q^2_N$ for $N$ large.}
\label{fig1}
\end{figure}

\begin{example}[Random walk on a singular graph]\cite{sc1, bl9}
\label{ex4} 
In this example, the metastable behavior is not due to an energy
landscape but to the presence of bottlenecks. After attaining a well,
the system remains there a time long enough to relax inside the well
before it hits a point from which it can jump to another well.  In
this example, to fulfil condition (M1) the set $\ms B^x_N$ can not be
taken as a singleton.

In many other models the entropy plays an important role in the
metastable behavior. In the majority of them, the time-scale in which
the metastable behavior is observed can not be computed explicitly
and is given in terms of the spectral gap or the expectation of
hitting times. This is the case of polymers in the depinned
phase \cite{cmt, clmst, LacTei}, or the evolution of a droplet in the
Ising model with the Kawasaki dynamics \cite{bl2015a, gl2015}.

We consider below a random walk on a graph $E_N$ which is illustrated
in Figure \ref{fig1} in the two-dimensional case.  For $N\ge 1$, $d\ge
2$, let $I_N = \{0, \dots, N\}$, $Q^+_N = I^2_N \times I^{d-2}_N$,
$Q^-_N = I^2_N \times (-I_N)^{d-2}$ be $d$-dimensional cubes of length
$N$. Let $w_i = w^N_i$, $0\le i\le 3$, be the points in $\bb Z^d$
given by $w_0 = (0,N, \bs 0)$, $w_1 = (N,0,\bs 0)$, $w_2 = (0,-N, \bs
0)$, $w_3 = (-N,0, \bs 0)$, where $\bs 0$ is the $(d-2)$-dimensional
vector with all coordinates equal to $0$.  Set $Q^i_N = w_i + Q^+_N$,
$i=0$, $2$, $Q^j_N = w_j + Q^-_N$, $j=1$, $3$, $E_N = \sqcup_{0\le i\le
  3} Q^i_N$. Note that the sets $Q^i_N \cap Q^{i+1}_N$ are singletons
in all dimensions. This explains the rather intricate definition of
the sets $Q^i_N$.

Denote by $e_1, \dots, e_d$ the canonical basis of $\bb R^d$.  Let
$\eta^N(t)$ be the continuous-time Markov chain on $E_N$ which jumps
from a configuration $\eta \in E_N$ to $\eta \pm e_j \in E_N$ at rate
$1$ if $\eta \mp e_j\in E_N$ and at rate $2$ if $\eta \mp e_j\not\in
E_N$. With these jump rates the Markov chain on the cube $I^d_N$ can
be thought as the projection on $I^d_N$ of a simple random walk on
$\bb Z^d$.

Denote by $n(\eta)\in \{0, 1, \dots, d\}$, $\eta\in E_N$, the number
of neighbors of $\eta$ which do not belong to $E_N$, and by $\ms C$
the four corners of $E_N$: $\ms C = \{\eta\in E_N : \eta \in Q^x_N
\cap Q^y_N \text{ for some } x\not = y\}$. Let $\mu_N$ be the
probability measure on $E_N$ given by 
\begin{equation*}
\mu_N(\eta) \;=\; \frac 1{Z_N}\,  \frac 1{2^{n(\eta)}} \;, \quad \eta\not\in \ms
C\;, \quad
\mu_N(\xi) \;=\; \frac 1{Z_N}\,  \frac 1{2^{d-1}} \;, \quad \xi \in \ms
C\;,
\end{equation*}
where $Z_N$ is the normalizing factor. The measure $\mu_N$ is the
unique stationary (actually, reversible) state. Denote by $\theta_N$
the inverse of the spectral gap of this chain. By \cite[Example
3.2.5]{sc1}, there exist constants $0<c(d) < C(d)<\infty$ such that
for all $N\ge 1$,
\begin{equation*}
c(2) \, N^2 \,\log N \;\le\; \theta_N  \;\le\;C(2)\, N^2
  \,\log N \;, \quad d=2
\end{equation*}
and
\begin{equation*}
c(d)\, N^d \;\le\; \theta_N  \;\le\; C(d)\, N^d\; \quad d\ge 3.
\end{equation*}
\end{example}
Fix sequences $\{\ell_N : N\ge 1\}$, $\{M_N : N\ge 1\}$, $1\ll \ell_N\ll
M_N \ll N$, such that
\begin{equation}
\label{l3-17}
\log\ell_N/\log N \to 1, \quad d=2\qquad\text{and}\qquad
N^2 \, \ell^{d-2}_N \,\ll\, M^d_N\;,\quad d\ge 3.
\end{equation}
Recall that we denote by $\ms C$ the four corners of $E_N$. Let
$\Delta_N$ be the points at graph distance less than $\ell_N$ from one
of the corners:
\begin{equation*}
\Delta_N \;=\; \{\eta\in E_N : d(\eta, \ms C) \le \ell_N\}\;,
\end{equation*}
where $d(\eta,\xi)$ stands for the graph distance from $\eta$ to
$\xi$. Finally, let $\ms E^x_N = Q^x_N \setminus \Delta_N$, $J_N = \{
M_N , \dots, N-M_N\}$, and $\ms B^x_N = w_x + J_N^2\times \big( (-1)^x
J_N\big)^{d-2}$. Note that $\ms B^x_N \subset \ms E^x_N$.  We refer to
Figure \ref{fig1} for an illustration of these sets.

Assumptions (H1) and (H2) for this model follow from the
arguments presented in \cite[Proposition 8.3]{bl9}.  Condition
\eqref{03} follows from Lemma \ref{l04}.

\smallskip\noindent{\sl A. First condition of Lemma \ref{l04}.}  If
$d\ge 3$, this condition follows easily from Lemma \ref{l06a}. Indeed,
since the mixing time of a random walk on a $d$-dimensional cube of
length $N$ is of order $N^2$, condition \eqref{tracemixing2} is an
easy consequence of \eqref{tracemixestimate}.  The following argument
also works for $d=2$.

Fix $\delta>0$, $\eta\in \ms E^0_N$, and recall that we denote by $\ms
C$ the set of corners. Let $\varepsilon_N\ll 1$ be a sequence such
that $N^2 \ll \varepsilon_N\, \theta_N$.  By equation (6.18) in
\cite{jlt1},
\begin{equation}
\label{l3-18}
\lim_{N\to \infty} \max_{\xi\in \ms E_N} 
\bb P_\xi \big[ H_{\ms C} \le \varepsilon_N \big] \;=\; 0\;.
\end{equation}
We may therefore assume that the process $\xi^N(t)$ does not hit $\ms
C$ before $\varepsilon_N$. On this event, we may couple $\xi^N(t)$ with a
speeded-up random walk $\widehat{\xi}_N(t)$ on $I^d_N$, and $\xi^N(t)$ hits $\ms B_N^x$ 
when $\widehat{\xi}_N(t)$ hits $J_N^d$. By Theorem 5 in \cite{aldous} applied to $\widehat{\xi}_N(t)$,
\[
\mu_N^x(\ms B_N^x)\sup_{\xi\in\ms E_N^x}\bb E_\xi \big[H_{\ms B_N^x};\ H_{\ms C}>\varepsilon_N\big]\le C_0N^2\theta_N^{-1}.
\]

Since $\mu_N^x(\ms B_N^x)\ge c_0>0$ and $\theta_N\varepsilon_N \gg
N^2$, this proves that
\[
\lim_{N\to\infty}\sup_{\xi\in\ms E_N^x}\bb P_\xi\big[H_{\ms B_N^x}>\varepsilon_N\big]=0,
\]
and in particular the first condition of Lemma \ref{l04}.

\smallskip\noindent{\sl B. Second condition of Lemma \ref{l04}.}  The
argument is based on the fact that the process relaxes to equilibrium
inside each cube much before it hits the corners.  Fix $\delta>0$,
$\delta<s<3\delta$, $\eta\in \ms E^0_N$, and let $\varepsilon_N$ be as
in A, i.e. $N^2 \ll \varepsilon_N\, \theta_N \ll \theta_N$.
By \eqref{l3-18}, we may insert the
event $\{H_{\ms C} > \varepsilon_N \}$ inside the probability
appearing in the second displayed equation in Lemma \ref{l04}. After
this operation, applying the Markov property, the probability becomes
\begin{equation*}
\bb E_\eta \Big[ \mb 1\{ H_{\ms C} > \varepsilon_N  \}\, 
\bb P_{\xi^N(\varepsilon)}  \big[ \xi^N(s-\varepsilon_N ) \in \Delta_N
\big]\, \Big]\;.
\end{equation*}
On the set $\{ H_{\ms C} > \varepsilon_N \}$, we may couple the process
$\xi^N(t)$ with the speeded-up, random walk reflected at
$Q^0_N$. Denote by $\bb P^0_N$ the distribution with respect to this
dynamics and by $\bb E^0_N$ the expectation.

Up to this point we proved that
\begin{equation*}
\limsup_{N\to\infty} \bb P_\eta \big[ \xi^N(s) \in \Delta_N \big]
\;\le\;
\limsup_{N\to\infty} \bb E^0_\eta \Big[ 
\bb P_{\xi^N(\varepsilon_N)}  \big[ \xi^N(s-\varepsilon_N ) \in \Delta_N
\big]\, \Big]\;.
\end{equation*}
Since the mixing time of the (speeded-up) random walk on $Q^0_N$ is of
order $N^2/\theta_N \ll \varepsilon_N$, the previous expression is
bounded by
\begin{equation*}
\limsup_{N\to\infty} \bb P_{\mu^0_N}  \big[ \xi^N(s-\varepsilon_N ) \in \Delta_N
\big] \;,
\end{equation*}
where $\mu^0_N$ is the stationary state of the reflected random
walk. As $\mu^0_N (\eta) \le C_0 \mu_N (\eta)$, and since $\mu_N$ is
the stationary state, the previous expression is bounded by
\begin{equation*}
C_0 \limsup_{N\to\infty} \bb P_{\mu_N}  \big[ \xi^N(s-\varepsilon_N ) \in \Delta_N
\big] \;=\;
C_0 \limsup_{N\to\infty} \mu_N  [  \Delta_N ] \; =\; 0\;,
\end{equation*}
which completes the proof of the second condition of Lemma \ref{l04}.

\medskip

The convergence of the finite-dimensional distributions has been
addressed in \cite{bl9}. We now turn to the assumptions of Proposition \ref{p02}. 
Condition (M1) has been proved above in A. We show below that (M2) is in force in dimension $d\ge 3$.

\smallskip\noindent{\sl C. Condition \eqref{b04}.} Recall from
\eqref{l3-17} that $N^2 \ll M^d_N/\ell^{d-2}_N$.  Let $\varepsilon_N$
be a sequence such that $N^2 \ll \varepsilon_N \, \theta_N \ll
M^d_N/\ell^{d-2}_N$.

Fix $\eta\in \ms B^0_N$. Up to the hitting time of the set $\Delta_N$
the process $\xi^N(t)$ behaves as the chain $\widehat{\xi}_N(t)$
introduced below \eqref{l3-18}. It is therefore enough to prove
condition \eqref{b04} for this latter process. Let $\Delta^{(1)}_N$,
$\Delta^{(2)}_N$ be the simplexes given by
\begin{align*}
& \Delta^{(1)}_N \;=\; \{x\in \bb Z^d : x_i \ge 0 \,,\,
\sum_i x_i \le \ell_N\}\;, \\
&\quad \Delta^{(2)}_N \;=\; (N,0,\bs 0) + 
\{(y,x)\in \bb Z \times \bb Z^{d-1} : y\le 0\,,\, x_i \ge 0 \,,\,
-y + \sum_i x_i \le \ell_N\}\;.
\end{align*}
We have to show that for $i=1$, $2$,
\begin{equation}
\label{l3-14}
\lim_{N\to \infty} \max_{\eta\in {J_N^d}}
\mb P_\eta \big[ H_{\Delta^{(i)}_N} \le \varepsilon_N \big] \;=\; 0\;,
\end{equation}
where $\mb P_\eta$ stands for the distribution of $\widehat \xi_N(t)$ starting from $\eta$. 
By symmetry, it suffices to do so for $i=1$.

Set $\gamma_N = \varepsilon^{-1}_N$, and denote by
$\zeta^\star_N(t)$ the $\gamma_N$-enlargement of the process $\widehat
\xi_N(t)$. We refer to Section \ref{sec04} for the definition of the
enlargement and the statement of some properties. Denote by $\mb
P^\star_\eta$ the distribution of the process $\zeta^\star_N(t)$
starting from $\eta$, and by $V^\star$ the equilibrium potential
between $\Delta^{(1)}_N$ and $\mb E^\star_N$: $V^\star(\eta) = \mb
P^\star_\eta \big[ H_{\Delta^{(1)}_N} \le H_{\mb E^\star_N} \big]$.
By \eqref{lv3-02}, \eqref{l3-14} follows from
\begin{equation}
\label{l3-16}
\lim_{N\to \infty} \max_{\eta\in {J_N^d}} V^\star(\eta) \;=\; 0\;.
\end{equation}

To bound the equilibrium potential $V^\star$, we follow a strategy
proposed in \cite{bl9}. We first claim that
\begin{equation}
\label{l3-15}
\Cap^\star_N (\Delta^{(1)}_N, \mb E^\star_N) \;\le\; 
\frac{ C_0\, \ell^d_N }{N^d}\, \Big\{ \frac {\theta_N}{\ell^2_N} \,+\,
\gamma_N\Big\} \;, 
\end{equation}

Fix $L_N = 2\ell_N$, and let $f: \bb N \to \bb R_+$ the function given
by $f(k) = 1$ for $0\le k< \ell_N$, $f(k) = 0$ for $k\ge L_N$ and $f(k) =
A \sum_{k\le j <L_N} j^{-(d-1)}$ for $\ell_N \le k<L_N$, where $A$ is chosen
for $f(\ell_N)=1$. Let $F:\mb E_N \to \bb R$, $F^\star:\mb E_N \sqcup \mb
E^\star_N \to \bb R$ be given by $F(x) = f(\sum_{1\le i\le d} x_i)$,
$F^\star(\eta) = F(\eta)$, $\eta\in \mb E_N$, $F^\star(\eta) = 0$,
$\eta\in \mb E^\star_N$. By the Dirichlet principle, $\Cap^\star_N
(\Delta^{(1)}_N, \mb E^\star_N) \le D^\star_N(F^\star)$, where
$D^\star_N$ represents the Dirichlet form of the enlarged process
$\zeta^\star_N(t)$.

There are two contributions to the Dirichlet form
$D^\star_N(F^\star)$. The first one corresponds to edges 
whose vertices belong to the set $
\Lambda_N = \{x \in \mb E_N : \ell_N \le \sum_i x_i \le L_N\}$. This
contribution is bounded by
\begin{equation*}
\frac{C_0\, \theta_N}{N^d} \sum_{j=\ell_N}^{L_N} j^{d-1} \big[ f(j+1) - f(j)\big]^2
\;\le\; \frac{C_0 \, \theta_N\, \ell^{d-2}_N}{N^d}\;. 
\end{equation*}
The other contribution, is due to the edges between the sets $\Lambda_N$
and $\Lambda^\star_N$. Since $F^\star$ is bounded by 1, this contribution
is bounded by $\frac{1}{4}\gamma_N\mu_N(\Lambda_N)\le C_0 \gamma_N \ell^d_N
/N^d$.  This completes the proof of \eqref{l3-15}. 

We turn to \eqref{l3-16}. Let $\prec$ be the partial order on $J_N^d$ 
defined by $\eta\prec \xi$ if $\eta_i \le \xi_i$ for $1\le i\le
d$. We may couple two copies of the process $\widehat \xi_N(t)$,
denoted by $\zeta^\eta_N(t)$, $\zeta^\xi_N(t)$, starting from
$\eta\prec\xi$, respectively, in such a way that $\zeta^\eta_N(t)
\prec \zeta^\xi_N(t)$ for all $t\ge 0$. In particular,
$\zeta^\eta_N(t)$ hits $\Delta^{(1)}_N$ before $\zeta^\xi_N(t)$, so
that
\begin{equation*}
V^\star(\eta) \;=\; \mb P^\star_\eta \big[ H_{\Delta^{(1)}_N} \le
H_{\mb E^\star_N} \big] \;\ge\; \mb P^\star_\xi \big[ H_{\Delta^{(1)}_N} \le
H_{\mb E^\star_N} \big] \;=\; V^\star(\xi)\;.
\end{equation*}
Suppose that \eqref{l3-16} does not hold. There exists, therefore,
$\delta>0$, a subsequence $N_j$, still denoted by $N$, and a
configuration $\eta^N\in{J_N^d}$ such that $V^\star(\eta^N)\ge
\delta$. By the previous inequality and by definition of ${J_N^d}$,
$V^\star(\xi)\ge \delta$ for all $\xi$ such that $\max_i \xi_i \le
M_N$. In particular,
\begin{equation*}
\Cap^\star_N (\Delta^{(1)}_N, \mb E^\star_N)  \;=\; D^\star_N(V^\star) 
\;\ge\; c_0\, M^d_N \, \frac{\gamma_N}{N^d}\, \delta^2\;.
\end{equation*}
Comparing this bound with \eqref{l3-15} we deduce that $\delta^2 \,
\gamma_N \, M^d_N \le C_0 \ell^{d-2}_N\, \theta_N$, which is a
contradiction since $\gamma_N = \varepsilon^{-1}_N$ and $\varepsilon_N
\, \theta_N \ll M^d_N/\ell^{d-2}_N$.

\smallskip\noindent{\sl D. Condition \eqref{b07}.}  It is well known
that the mixing time of a random walk on a $d$-dimensional cube of
length $N$ is of order $N^2$, which proves that condition \eqref{b07}
is fulfilled since $\varepsilon_N \, \theta_N \gg N^2$.

\smallskip\noindent{\sl E. Last conditions of Proposition \ref{p02}.}
Condition \eqref{b09a} is clearly in force by definition of
$\Delta_N$. On the other hand the chain is reversible.

\section{Appendix}
\label{sec04}

We present in this section a general estimate for the hitting time of
a set in Markovian dynamics. Fix a finite set $E$ and let $\{\eta(t) :
t\ge 0\}$ be a continuous-time, irreducible, $E$-valued Markov
chain. Denote by $\pi$ the unique stationary state of the process, by
$R(\eta, \xi)$, $\eta$, $\xi\in E$ its jump rates, and by $\bb P_\eta$
its distribution starting from $\eta$.

We start with an elementary lemma.

\begin{lemma}
\label{lv3-l2}
Let $X$, $T_\gamma$ be two independent random variables defined on some
probability space $(\Omega, \mc F, P)$. Assume that $T_\gamma$ has an
exponential distribution of parameter $\gamma>0$. Then, for all $b>0$,
\begin{equation*}
P \big[ X \le b \big] \;\le\; e^{\gamma b} \, P \big[ X \le T_\gamma
\big]\;. 
\end{equation*}
\end{lemma}

\begin{proof}
Since $X$ and $T_\gamma$ are independent, for every $b>0$,
\begin{equation*}
P \big[ X \le T_\gamma \big] \;\ge\;
\int_{b}^\infty P \big[ X \le t \big] \, \gamma\, e^{-\gamma t}\,
dt \;\ge\; P \big[ X \le b \big] \, 
\int_{b}^\infty \gamma\, e^{-\gamma t}\, dt\;.
\end{equation*}
The last term is equal to $ e^{-\gamma b} \, P \big[ X \le b \big]$,
which completes the proof of the lemma. 
\end{proof}

Note that if $X$ is an exponential random variable of parameter
$\theta$, the inequality reduces to
\begin{equation*}
1\,-\, e^{-\theta b} \;\le\; e^{\gamma b}\, \frac \theta {\theta +
  \gamma}\;\cdot
\end{equation*}
Hence, choosing $\gamma=1/b$, if $\theta b$ is small, the inequality
is sharp in the sense that the left-hand side is equal to $\theta\, b
+ O([\theta\, b]^2)$, while the right-hand side is equal to $e \,
\theta\, b + O([\theta\, b]^2)$.

\smallskip\noindent{\bf Enlargement of a chain} \cite{biagau2016,
  bl9}. Let $E^\star$ be a copy of $E$ and denote by $\eta^\star\in
E^\star$ the copy of $\eta\in E$. Denote by $\xi^\gamma (t)$, $\gamma
>0$, the Markov process on $E \sqcup E^\star$ whose jump rates $R^\gamma
(\eta,\xi)$ are given by
\begin{equation*}
R^\gamma (\eta,\xi) \;=\; 
\begin{cases}
R(\eta,\xi) & \text{if $\eta$ and $\xi\in E$,} \\
\gamma & \text{if $\xi = \eta^\star$ or $\eta = \xi^\star$,} \\
0 & \text{otherwise.} 
\end{cases}
\end{equation*}
Hence, being at some state $\xi^\star$ in $E^\star$, the process
may only jump to $\xi$ and this happens at rate $\gamma$. In contrast,
being at some state $\xi$ in $E$, the process $\xi^\gamma (t)$ jumps
with rate $R (\xi, \xi')$ to some state $\xi'\in E$, and jumps with
rate $\gamma$ to $\xi^\star$.  We call the process $\xi^\gamma (t)$
the \emph{$\gamma$-enlargement} of the process $\xi(t)$. Note that the
trace of the enlargement $\xi^\gamma (t)$ on $E$ coincides with the
original process $\xi (t)$.

The chain $\xi^\gamma (t)$ is clearly irreducible and its invariant
probability measure, denoted by $\pi^\star$, is given by 
\begin{equation}
\label{l3-05}
\pi^\star(\xi) \;=\; (1/2)\, \pi(\xi) \;, \quad \xi\in E\;, \quad
\pi^\star(\xi^\star) \;=\; \pi^\star(\xi)\;, \quad \xi^\star\in E^\star\;.
\end{equation}
The process $\xi^\gamma (t)$ reversed in time is the Markov chain,
denoted by $\xi^{\gamma, *} (t)$, whose jump rates $R^{\gamma, *}$ are
given by
\begin{equation*}
R^{\gamma, *} (\eta,\xi) \;=\; 
\begin{cases}
R^*(\eta,\xi) & \text{if $\eta$ and $\xi\in E$,} \\
\gamma & \text{if $\xi = \eta^\star$ or $\eta = \xi^\star$,} \\
0 & \text{otherwise,} 
\end{cases}
\end{equation*}
where $R^*(\eta,\xi)$ represents the jump rates of the process
$\xi(t)$ reversed in time.

Denote by $\bb P^\star_\eta$ the distribution of the chain $\xi^\gamma
(t)$ starting from $\eta$, and by $\Cap^\star (\ms C, \ms D)$ the
capacity between two disjoint subsets $\ms C$, $\ms D$ of $E \sqcup
E^\star$. 
\begin{lemma}
Fix two disjoint subsets $\ms A$, $\ms B$ of $E$. Then
\begin{equation}
\label{l3-20}
\Cap^\star (\ms A, \ms B) \;=\; (1/2) \, \Cap (\ms A, \ms B)\;. 
\end{equation}
and
\begin{equation}
\label{l3-20b}
\Cap^\star (\ms A, \ms B \sqcup E^\star) \;\ge\; (1/2) \, \big(\pi(\ms A)\gamma+\Cap (\ms A, \ms B)\big)\;. 
\end{equation}
\end{lemma}
\begin{proof}
By equation (2.6) in \cite{gaudl2014}, 
\begin{equation*}
\Cap^\star (\ms A, \ms B) \;=\; 
D^\star (V^\star_{\ms A, \ms B}) \;=\; 
\frac 12 \sum_{\eta , \xi \in E \sqcup E^\star} \pi^\star(\eta)\,
R^\star(\eta,\xi)\, [V^\star_{\ms A, \ms B} (\xi) - 
V^\star_{\ms A, \ms B} (\eta)]^2\;,
\end{equation*}
where $D^\star(f)$ represents the Dirichlet form of a function $f: E
\sqcup E^\star\to\bb R$ for the enlarged process, and $V^\star_{\ms C,
  \ms D}$ the equilibrium potential between two disjoints subsets $\ms
C$, $\ms D$ of $E \sqcup E^\star$: $V^\star_{\ms C, \ms D} (\eta) = \bb
P^\star_\eta[H_{\ms C} < H_{\ms D}]$. On the one hand, by definition
of the enlargement, for every $\eta\in E$, $V^\star_{\ms A, \ms B}
(\eta^\star) = V^\star_{\ms A, \ms B} (\eta)$. Hence, the contribution
to the Dirichlet form $D^\star(V^\star_{\ms A, \ms B})$ of the edges
between $E$ and $E^\star$ vanishes.  On the other hand, since the
trace of the enlargement $\xi^\gamma (t)$ on $E$ coincides with the
original process $\xi (t)$, for all $\eta\in E$, $V^\star_{\ms A, \ms
  B} (\eta) = \bb P^\star_\eta[H_{\ms A} < H_{\ms B}] = \bb
P_\eta[H_{\ms A} < H_{\ms B}] = V_{\ms A, \ms B} (\eta)$. Hence, the
sum appearing on the right-hand side of the previous displayed
equation is equal to
\begin{equation*}
\frac 12 \sum_{\eta , \xi \in E} \pi^\star(\eta)\,
R^\star(\eta,\xi) [V_{\ms A, \ms B} (\xi) - 
V_{\ms A, \ms B} (\eta)]^2\;.
\end{equation*}
Since, for $\eta$, $\xi \in E$, $R^\star(\eta,\xi) = R(\eta,\xi)$,
$\pi^\star(\eta) = (1/2) \pi (\eta)$, the previous sum is equal to
\begin{equation*}
\frac 14 \sum_{\eta , \xi \in E} \pi(\eta)\,
R (\eta,\xi) [V_{\ms A, \ms B} (\xi) - 
V_{\ms A, \ms B} (\eta)]^2 \;=\; \frac 12 \, D (V_{\ms A, \ms B})
\;=\; \frac 12 \, \Cap (\ms A, \ms B) \;,
\end{equation*}
as claimed in \eqref{l3-20}.\\[2mm]
Let $\ms A^\star=\{\xi^\star\in E^\star:\ \xi\in \ms A\}$ and
$\lambda^\star(\eta)$ stand for the holding rate of $\xi^\gamma(t)$ at
$\eta$. We have
\begin{align*}
\Cap^\star (\ms A, \ms B\sqcup \ms A^\star)&=\sum_{\eta\in \ms A}\pi^\star(\eta)\lambda^\star(\eta)\bb P^\star_\eta\big[ H_{\ms B\sqcup \ms A^\star}<H_{\ms A}^+\big]\\
&=\frac{1}{2}\sum_{\eta\in\ms A} \pi(\eta)\sum_{\xi\in E\sqcup E^\star} R^{\gamma}(\eta,\xi)\bb P^\star_\xi\big[ H_{\ms B\sqcup \ms A^\star}<H_{\ms A}\big]\\
&=\frac{1}{2}\gamma\pi(\ms A)+\frac{1}{2}\sum_{\eta\in\ms A} \pi(\eta)\sum_{\xi\in E} R(\eta,\xi)\bb P^\star_\xi\big[ H_{\ms B\sqcup \ms A^\star}<H_{\ms A}\big],
\end{align*}
where in the last equality we have split the inner sum over $\xi \in \ms A^\star$ and $\xi \in E$. Taking into account that for every $\xi\in E$ we have $\bb P^\star_\xi\big[H_{\ms A^\star}>H_{\ms A}\big]=1$ because points $\eta^\star\in\ms A^\star$ are only accessible from $\eta\in\ms A$, the preceding computation gives
\begin{align*}
\Cap^\star (\ms A, \ms B\sqcup \ms A^\star) &=\frac{1}{2}\gamma\pi(\ms A)+\frac{1}{2}\sum_{\eta\in\ms A} \pi(\eta)\sum_{\xi\in E} R(\eta,\xi)\bb P^\star_\xi\big[ H_{\ms B}<H_{\ms A}\big]\\
&=\frac{1}{2}\gamma\pi(\ms A)+\frac{1}{2}\sum_{\eta\in\ms A} \pi(\eta)\sum_{\xi\in E} R(\eta,\xi)\bb P_\xi\big[ H_{\ms B}<H_{\ms A}\big]\\
&=\frac{1}{2}\gamma\pi(\ms A)+\frac{1}{2}\Cap(\ms A, \ms B).
\end{align*}
Inequality \eqref{l3-20b} now follows by monotonicity of capacities.
\end{proof}
Denote by $\nu^\star_{\ms A, \ms B}$ the equilibrium measure between
$\ms A$, $\ms B$ for the chain $\xi^\gamma (t)$, which is concentrated
on the set $\ms A$ and is given by
\begin{equation}
\label{lv3-03}
\nu^\star_{\ms A, \ms B} (\eta) \;=\; \frac{1}{\Cap^\star (\ms A, \ms B)}
\, \pi^\star (\eta)\, \lambda^\star (\eta) \, \bb P^\star_\eta \big[
H_{\ms B} < H^+_{\ms A}\big] \;.
\end{equation}
If $\ms A$ is a set with small measure with respect to the stationary
measure, it is expected that, for most configurations $\eta\in E$,
$H_{\ms A}$ is approximately exponentially distributed under $\bb
P_\eta$. Let $\lambda^{-1}$ be its expectation, so that $\bb P_\eta
\big[ H_{\ms A} \le b \big] \approx 1 - \exp\{- b \lambda\} \approx b
\lambda$, provided $b\lambda \ll 1$. On the one hand, by
\cite[Proposition A.2]{bl4},
\begin{equation*}
\lambda^{-1} \;\approx\; \bb E_\eta \big[ H_{\ms A} \big] \;=\;
\frac{ \< V^*_{\eta, \ms A} \>_\pi}{\Cap (\eta, \ms A)} \;,
\end{equation*}
where $V^*_{\eta, \ms A}$ is the equilibrium potential between $\eta$
and $\ms A$ for the time-reversed dynamics, and $\Cap (\eta, \ms A)$
the capacity between $\eta$ and $\ms A$.  If $\< V^*_{\eta, \ms A}
\>_\pi \approx 1$ (for instance, because $\pi(\eta)\approx 1$), we
conclude that $\lambda \approx \Cap (\eta, \ms A)$. On the other hand,
choosing $\gamma = b^{-1}$ as the parameter for the enlarged process,
for every $\eta\in E$,
\begin{equation*}
b \;=\; \gamma^{-1} \;=\; \bb E^\star_\eta \big[ H_{E^\star} \big]
\;=\; \frac{ \< V^{\star,*}_{\eta, E^\star} \>_{\pi^\star}}{\Cap^\star (\eta,E^\star)} \;\cdot    
\end{equation*}
Once more, if $\< V^{\star,*}_{\eta,E^\star} \>_{\pi^\star} \approx
1$, we conclude that $b^{-1} \approx \Cap^\star (\eta, E^\star)$, so
that
\begin{equation*}
\bb P_\eta \big[ H_{\ms A} \le b \big] \;\approx\; b\lambda  \;\approx\;
\frac {\Cap (\eta, \ms A)}{\Cap^\star (\eta, E^\star)}.
\end{equation*}
The next lemma establishes this estimate. 
\begin{lemma}
\label{lv3-l1}
Fix a proper subset $\ms A$ of $E$. For every $b>0$ and $\eta\in E
\setminus \ms A$,
\begin{equation*}
\bb P_\eta \big[ H_{\ms A} \le b \big] \;\le\; 
\frac{1}{2}e^{\gamma b} \, \frac{\Cap (\eta , \ms A)}
{\Cap^\star (\eta ,  \ms  A\sqcup E^\star)} \;, 
\end{equation*}
and
\begin{equation*}
\bb P_\eta \big[ H_{\ms A} \le b \big] \;\le\; 
e^{\gamma b} \, \frac 1{\gamma\, \pi^\star(\eta)} 
\, \Cap^\star (\ms A, E^\star)\;.
\end{equation*}
\end{lemma}

\begin{proof}
Fix a proper subset $\ms A$ of $E$, $b>0$ and $\eta\in E \setminus \ms
A$. Fix $\gamma>0$, and consider the $\gamma$-enlarged
process. Denote by $H_{E^\star}$ the hitting time of the set
$E^\star$. By definition of the enlargement, under $\bb P^\star_\eta$,
$H_{E^\star}$ has an exponential distribution of parameter $\gamma$ and is
independent of $H_{\ms A}$. Hence, by Lemma \ref{lv3-l2},
\begin{equation}
\label{lv3-02}
\bb P_\eta \big[ H_{\ms A} \le b \big] \;\le\; e^{\gamma b} \,
\bb P^\star_\eta \big[ H_{\ms A} \le  H_{E^\star} \big]\;.
\end{equation}
The previous probability is the value of the equilibrium potential
between $\ms A$ and $E^\star$ computed at the configuration $\eta$,
denoted hereafter by $V^\star_{\ms A, E^\star}$. By equation (3.3) in
\cite{ll} and by \eqref{l3-20}, the previous expression is bounded by
\begin{equation*}
e^{\gamma b} \, \frac{\Cap^\star (\eta , \ms A)}
{\Cap^\star (\eta ,  \ms  A\sqcup E^\star)} \;=\;
\frac{1}{2}e^{\gamma b} \, \frac{\Cap (\eta , \ms A)}
{\Cap^\star (\eta ,  \ms  A\sqcup E^\star)} \,\cdot
\end{equation*}
This proves the first assertion of the lemma.

We may also rewrite the right-hand side of \eqref{lv3-02} as
\begin{equation*}
e^{\gamma b} \, \frac 1{\pi^\star(\eta)}  \, \sum_{\zeta\in E\sqcup
  E^\star} V^\star_{\ms A, E^\star} (\zeta) \mb 1\{\eta\}(\zeta)\,
\pi^\star(\zeta)\;,
\end{equation*}
where $\mb 1\{\eta\}$ represents the indicator of the set $\{\eta\}$. 
By \cite[Proposition A.2]{bl4}, the previous sum is equal to
\begin{equation*}
\Cap^\star (\ms A, E^\star)\, \bb E^{\star,*}_{\nu_{\ms A, E^\star}} 
\Big[ \int_0^{H_{E^\star}} \mb 1\{\eta\}(\xi^*(t)) \, dt \Big]\;,
\end{equation*}
where $\bb P^{\star,*}$ represents the distribution of the process
$\xi^\gamma(t)$ reversed in time, and $\nu_{\ms A, E^\star}$ the
equilibrium measure given by \eqref{lv3-03}.  By definition of the
enlarged process, for every initial condition $\eta\in E$,
$H_{E^\star}$ has an exponential distribution of parameter
$\gamma$. The penultimate displayed equation is thus bounded by
$\gamma^{-1} \Cap^\star (\ms A, E^\star)$, which completes the proof
of the lemma.
\end{proof}

Denote by $\partial_+ \ms A$ the exterior boundary of a set $\ms A$:
\begin{equation*}
\partial_+ \ms A \;=\; \big \{\eta\in E\setminus \ms A : 
\pi(\xi) \, R(\xi, \eta) >0
\text{ for some $\xi\in\ms A$ } \big\}\;.
\end{equation*}

\begin{corollary}
\label{lv3-l3}
Fix a proper subset $\ms A$ of $E$. For every $b>0$ and $\eta\in E
\setminus \ms A$,
\begin{equation*}
\bb P_\eta \big[ H_{\ms A} \le b \big] \;\le\; 
 \, \frac {e\, b} {\pi(\eta)} \, \Cap (\eta, \ms A)
\;\le\; 
 \, \frac {e\, b} {2\pi(\eta)} 
\, \sum_{\xi\in \partial_+\ms A} \pi(\xi)\, R(\xi, \ms A)\;,
\end{equation*}
where $R(\xi, \ms A) = \sum_{\zeta\in\ms A} R(\xi,\zeta)$.
\end{corollary}
\begin{proof}
In view of \eqref{l3-20b}, the first result of the preceding lemma gives
\[
\bb P_\eta \big[ H_{\ms A} \le b \big] \;\le\; e^{\gamma b}\frac{\Cap (\eta, \ms A)}{\pi(\eta)\gamma+\Cap (\eta, \ms A)}.
\]
It suffices now to pick $\gamma=b^{-1}$. For the second inequality note that 
\[
\Cap (\eta, \ms A)\le \Cap (E\setminus\ms A, \ms A)=\frac{1}{2}\sum_{\xi\in \partial_+\ms A} \pi(\xi)\, R(\xi, \ms A).
\]
\end{proof}

\noindent{\bf Acknowledgements}. C. Landim has been partially
supported by FAPERJ CNE E-26/201.207/2014, by CNPq Bolsa de
Produtividade em Pesquisa PQ 303538/2014-7, and by ANR-15-CE40-0020-01
LSD of the French National Research Agency. 

\bibliographystyle{plain}
\bibliography{cfdd}

\end{document}